\documentclass[letterpaper,12pt]{article}

\usepackage{fullpage}
\usepackage{suwarastyle}
\usepackage
  [hypertexnames=false,colorlinks=true,linkcolor=blue,citecolor=blue,urlcolor=blue]{hyperref}
\usepackage{suwaratools}
\usepackage{gaugesymbols}
\usepackage{coloredtheorems}

\namedtheorem{Semi\hyp{}infinite\hyp{}dimensionality Theorem}{semiinfdimthm}
\namedtheorem{Gluing Theorem}{gluingthm}
\namedtheorem{Regularity Theorem}{regularitythm}
\namedtheorem{Strong UCP Theorem}{strongucpthm}
\namedtheorem{Implicit Function Theorem}{iftheorem}

\title{Low-regularity Seiberg-Witten moduli spaces on manifolds with boundary}

\author{Piotr Suwara \thanks{Institute of Mathematics, 
		Polish Academy of Sciences,
		ul. Śniadeckich 8, 
		00-656 Warszawa,
	email: \href{mailto:psuwara@impan.pl}{psuwara@impan.pl}}}

\date{\today}

\begin{document}

\maketitle

\abstract{
	For a~compact \spinc{} manifold \(X\) with boundary
	\(b_1(\partial X)=0\),
	we consider 
	moduli spaces of solutions to the Seiberg\hyp{}Witten equations
	in a~generalized double Coulomb slice 
	in \(W^{1,2}\) Sobolev regularity.
	We prove they
	are Hilbert manifolds,
	prove denseness and ''semi\hyp{}infinite\hyp{}dimensionality''
	properties of the restriction to \(\partial X\),
	and establish a~gluing theorem.

	To achieve these, we prove
	a~general regularity theorem and
	a~strong unique continuation principle for Dirac operators,
	and smoothness of a~restriction map to configurations
	of higher regularity on the interior,
	all of which are of independent interest.
}

\tableofcontents

\section{Introduction}
This article is motivated by an effort to provide
a~new framework for constructing so\hyp{}called Floer theories,
which are used to construct invariants of low\hyp{}dimensional manifolds,
knots, as well as Lagrangians in symplectic manifolds.
Originally, Floer implanted ideas from finite\hyp{}dimensional Morse theory
(cf. \cites{Sch1993,Wit1982})
to the study of functions 
on infinite\hyp{}dimensional spaces
(often called \textit{functionals} on \textit{configuration spaces}), 
leading to the resolution of the Arnold conjecture in symplectic topology
\cites{Flo1987,Flo1988Lag}.
Inspired by Donaldson's proof of the diagonalization theorem
for \(4\)\hyp{}manifolds \cite{Don1983}, 
Floer used these ideas
to construct homological invariants of \(3\)\hyp{}manifolds,
aiming to establish gluing formulas for Donaldson's invariants.
An~analogous construction for knots in \(3\)\hyp{}manifolds
has been carried out by Kronheimer and Mrowka
and plays the key role in the proof that Khovanov homology
detects the unknot \cite{KM2011a}.
An analogue for the Seiberg\hyp{}Witten equations (in place of the
anti\hyp{}self\hyp{}duality (ASD) equations of Donaldson's theory)
is the central ingredient 
of Manolescu's disproof of the triangulation conjecture
in dimensions \(\geq 5\) \cite{Man2013}.

While the idea of using Morse theory in infinite\hyp{}dimensional settings dates back
to Atiyah and Bott \cite{AB1983},
Floer's novelty was dealing with 
functionals 
having critical points of infinite index.
They key observation was that the unstable manifolds of critical points
were of finite index with respect to a~certain subbundle of the tangent space,
allowing one to define finite relative indices of critical points
and eventually leading to a~computation of 
''middle\hyp{}dimensional'' homology groups of the space.
As noted by Atiyah \cite{A1988} (cf. \cite{CJS1995}),
Floer theory was, from the very beginning,
understood as describing the behavior of so\hyp{}called 
\emph{semi\hyp{}infinite\hyp{}dimensional cycles}.

As in Morse homology (cf. \cite{Sch1993}),
one needs to overcome analytical difficulties
to even define Floer theories and prove they are well\hyp{}defined
and independent of choices -- one needs to establish regularity and compactness 
of moduli spaces of trajectories between critical points.
Moreover, Morse theory is generally not well\hyp{}suited 
for equivariant constructions
since one in general cannot guarantee regularity of the moduli spaces
without breaking the symmetry coming from a~group action.
To address these problems
and to allow general constructions of equivariant 
(cf. \cites{DS2019pre,KMbook,FLin2018,SMMil2019})
and generalized Floer homologies
(cf. \cites{Man2014,JLin2015,AB2021pre}),
Lipyanskiy \cite{Lip2008} introduced a~framework for using such 
semi\hyp{}infinite\hyp{}dimensional cycles 
as a~tool for defining Floer theories
and the author has further developed these methods in \cite{Suw2020}.
This article contains results required to rigorously 
define the relative invariants of \(4\)\hyp{}manifolds with boundary
and maps induced by cobordisms
in this construction, using the Seiberg\hyp{}Witten equations.
Moreover, the key result is a~gluing theorem which 
is fundamental for establishing functoriality of the
cobordism maps.

\paragraph{Results.}
Consider a~\(4\)\hyp{}manifold \(X\) with boundary
a~nonempty collection of rational homology spheres
(i.e., \(b_1(\partial X) = 0\))
and a~\spinc{} structure \(\hat{\mathfrak{s}}\) on \(X\).
We study the moduli space of solutions to the Seiberg\hyp{}Witten equations
for pairs \((A,\Phi)\) of a~\spinc{} connection \(A\) on \(X\)
and a~section \(\Phi\) of the spinor bundle \(S^+\) over \(X\)
associated to the \spinc{} structure \(\hat{\mathfrak{s}}\):
\begin{equation*}
	\left\{
	\begin{aligned}
		\frac 1 2 F^+_{A^t} - \rho^{-1}((\Phi\Phi^\ast)_0) 
			&= 0, \\
		\Dirac_{A}^+ \Phi  &= 0.
	\end{aligned}
	\right.
\end{equation*}
(\autoref{def:Seiberg-Witten-equations}).
The split Coulomb slice with respect to a~reference connection \(A_0\)
is given by the equations
\begin{equation*}
	\left\{
		\begin{aligned}
			d^\ast (A-A_0) 
			&= 0, \\
			d^\ast (\iota_{\partial X}^\ast (A-A_0)) 
			&= 0,
		\end{aligned}
	\right.
\end{equation*}
(where \(\iota_{\partial X}: \partial X \hookrightarrow X\) denotes the inclusion)
together with a~condition restricting \(A-A_0\)
to a~subset of codimension \(b_0(\partial X) - 1\)
(\autoref{def:split-Coulomb-slice}).
This additional condition 
depends on a~choice of a~gauge splitting \(s\)
(\autoref{def:gauge-splitting}).
The moduli space \(\SWModuli{s}{X}\)
(\autoref{def:moduli-spaces-on-4-manifolds-with-boundary})
is the quotient of the space of \(L^2_1\)\hyp{}solutions
to the Seiberg\hyp{}Witten equations
in this split Coulomb slice
by the (discrete) action of the split gauge group
(\autoref{def:gauge-group-split});
this gauge group preserves the split Coulomb slice
as well as the set of solutions to the Seiberg\hyp{}Witten equations.
There is also a~residual action of \(S^1\) on \(\SWModuli{s}{X}\)
given by multiplication of the spinor component, \(\Phi \mapsto z \Phi\),
by complex numbers in the unit circle \(z \in S^1 \subset \mathbb{C}\).

Firstly, we choose a~gauge twisting \(\tau\)
(\autoref{def:gauge-twisting})
and define the (\(S^1\)\hyp{}equivariant) twisted restriction map
\(R_\tau : \SWModuli{s}{X} \to \CoulThree{\partial X}\),
taking values in the configurations on \(\partial X\)
in the Coulomb slice
(\autoref{def:Coulomb-slice}).
The gauge splittings and gauge twistings we introduce
generalize the double Coulomb slice used in
\cites{Lip2008,Kha2015} and twistings utilized in \cite{KLS2018}.
We prove
regularity, denseness and 
``semi\hyp{}infinite\hyp{}dimensionality''
of the Seiberg\hyp{}Witten moduli spaces:
\begin{restatable*}%
	{semiinfdimthm}{moduli}
	\label{thm:semi-infinite-dimensionality-of-moduli-spaces}
	The moduli spaces \(\SWModuli{s}{X}\) are Hilbert manifolds.
	The differential of the twisted restriction map 
	\(R_\tau : \SWModuli{s}{X} \to \CoulThree{\partial X}\)
	decomposes into 
	\(\proj^- D R_\tau:
	T \SWModuli{s}{X} \to H^-(\partial X,\mathfrak{s})\) 
	which is Fredholm
	and
	\(\proj^+ D R_\tau:
	T \SWModuli{s}{X} \to H^+(\partial X,\mathfrak{s})\) 
	which is compact.

	Moreover,
	if \(b_0(\partial X)>1\),
	then for any connected component
	\(Y_0 \subset \partial X\)
	the restriction \(R_{\tau,Y_0}\) to \(Y_0\) has dense
	differential.
\end{restatable*}
The maps \(\Pi^\pm\) come from a~decomposition
\(\CoulThree{\partial X} = H^+(\partial X, \mathfrak{s})
\oplus H^-(\partial X, \mathfrak{s})\) 
according to the eigenvalues
of the operator \((\star d, \Dirac_{B_0})\),
also called a~\emph{polarization}
(\autoref{def:polarization-on-Coulomb-slice}).

The proof 
utilizes the Atiyah\hyp{}Patodi\hyp{}Singer boundary value problem
for an~extended linearized Seiberg\hyp{}Witten operator
(\autoref{def:extended-DSW}).
Then we prove its properties 
transfer to the restriction 
to the Coulomb slice.

Our low regularity setting requires us to prove
a~regularity theorem 
(\namedref{thm:low-regularity})
for an~operator of the form \(D = D_0 + K : L^2_1 \to L^2\), 
where \(D_0\) is a~Dirac operator and \(K\) is any compact operator,
making it a~result of independent interest.
We also prove a~strong version of the unique
continuation principle for \(D\) (\namedref{thm:unique-cont-dirac-Lr}).

Secondly, we prove a~gluing theorem for a~composite cobordism.
Assume \(X\) splits as \(X = X_1 \cup_Y X_2\)
along a~rational homology sphere \(Y\).
If the auxiliary data of gauge splittings 
and gauge twistings 
are compatible in
a~suitable sense
(see
\autoref{prop:twistings-are-integral-splittings}
and 
\autoref{prop:integral-splittings-on-a-composite-cobordism}),
then the moduli space \(\SWModuli{s}{X}\)
can be recovered from the fiber product of
\(\SWModuli{s_1}{X_1}\) and \(\SWModuli{s_2}{X_2}\)
over the configuration space of \(Y\),
in a~way compatible with the twisted restriction maps:
\begin{restatable*}%
	{gluingthm}{gluing}
	\label{thm:composing-cobordisms}
	Assume \(s_{\mathbb{Z}}\) and \( (s_{1,\mathbb{Z}}, s_{2,\mathbb{Z}})\)
	are compatible.
	Then there is an~\(S^1\)\hyp{}equivariant diffeomorphism 
	\(F : \SWModuli{s}{X}
	\to \SWModuli[1]{s_1}{X_1} \times_{Y} \SWModuli[2]{s_2}{X_2}\)
	such that
	\(R_{\tau}\)
	is \(S^1\)\hyp{}equivariantly homotopic to
	\( \left(R_{\tau_1}
		\times_{Y}
	R_{\tau_2} \right) \circ F\).
\end{restatable*}

The proof uses the following fact of independent interest.
We show that the restriction map of solutions in \(\SWModuli{s}{X}\)
to a~submanifold in the interior of \(X\) is smooth as a~map
into a~configuration space of \textit{higher} regularity
(\autoref{thm:restriction-is-smooth-for-solutions}).
While it is easy to prove that its image lies
in the space of smooth configurations
(\autoref{lem:interior-smoothness-of-solutions}),
the proof of \textit{smoothness} of this map is non\hyp{}trivial.

\paragraph{Applications.}
The \namedref{thm:semi-infinite-dimensionality-of-moduli-spaces}
together with compactness of moduli spaces proved in \cite{KMbook} 
(with minor modifications to account for the double Coulomb slice
instead of the Coulomb\hyp{}Neumann slice used in \cite{KMbook})
show that the maps \(R_\tau : \SWModuli{s}{X} \to \CoulThree{Y}\)
are \textit{semi\hyp{}infinite\hyp{}dimensional cycles} in \(\CoulThree{\partial X}\)
as defined in \cites{Lip2008,Suw2020},
establishing the existence of relative Seiberg\hyp{}Witten invariants of \(X\).
If the boundary \(\partial X\) is connected, these do not depend on the choice
of an~integral splitting (\autoref{lem:uniqueness-of-integral-splittings}).
This also implies that cobordisms \(\partial W = - Y_1 \cup Y_2\)
induce \textit{correspondences} 
(defined in \cites{Lip2008,Suw2020})
between the configuration spaces
over \(Y_1\) and \(Y_2\).

Our methods apply to perturbed equations
as well, which we did not include for the sake of simplicity.
Varying the metrics and perturbations gives cobordisms
between the relevant moduli spaces on \(X\),
and changing the reference connections on \(X\)
gives isomorphic moduli spaces with homotopic restriction maps.
This means that the relative invariant of \(X\),
up to a~suitable cobordism relation,
does not depend on the choices of perturbations and metric.

Crucially, the \namedref{thm:composing-cobordisms} says that
the correspondence induced by a~composite cobordism
is (homotopic to) the composition of the respective correspondences,
proving the theory comes with a~TQFT\hyp{}like structure.
We hope to prove that this theory
recovers a~non\hyp{}equivariant version of monopole Floer homology
\(\widetilde{HM}_\ast\)
and describe a~construction supposed to recover
all of the flavors of monopole Floer homology (for rational homology spheres) in an~upcoming work.

Applications of Seiberg\hyp{}Witten\hyp{}Floer theory
suggest including the case \(b_1(\partial X) > 0\)
or allowing \(X\) to have cylindrical or conical ends
(like in \cite{KM1997})
with certain asymptotic conditions on the configurations.
The methods presented here should suffice to prove 
regularity and semi\hyp{}infinite\hyp{}dimensionality of the corresponding
moduli spaces.
However, to establish the semi\hyp{}infinite\hyp{}dimensional theory
in full one needs to 
carefully select a~slice for the gauge action
and deal with compactness issues
which we hope to address in further work.

Finally, the methods presented here should be applicable 
(with appropriate adjustments) for defining
semi\hyp{}infinite\hyp{}dimensional Floer theories
in other contexts, e.g., in the Yang\hyp{}Mills\hyp{}Floer theory.

\paragraph{Organization.}
In \autoref{sec:analytical-preparation} we prove
the regularity theorem for a~Dirac operator with compact perturbation,
the \namedref{thm:low-regularity},
and a~strong unique continuation principle for a~Dirac operator with a~low\hyp{}regularity
potential, the \namedref{thm:unique-cont-dirac-Lr}.

In \autoref{sec:Seiberg-Witten-moduli-spaces-in-split-Coulomb-slice}
we introduce the basic notions of Seiberg\hyp{}Witten theory on \(3\)\hyp{} and \(4\)\hyp{}manifolds.
For a~chosen gauge splitting
the split Coulomb slice is defined,
and for a~gauge twisting a~twisted restriction map to the boundary
is introduced.
The choices of splittings and twistings are shown to be equivalent to a~choice
of an~integral splitting, one which does not require any twisting.

\Cref{sec:properties-of-moduli-spaces} provides proofs of some of the key
properties of the moduli spaces: regularity,
semi\hyp{}infinite\hyp{}dimensionality and denseness of the restriction map
(\namedref{thm:semi-infinite-dimensionality-of-moduli-spaces}).
Finally, \autoref{sec:gluing-along-a-boundary-component} 
contains the proof of the \namedref{thm:composing-cobordisms},
showing that moduli on a~composite cobordism
are a~fiber product of the moduli on its components.

\paragraph{Notations.}
In this article we use the notation \(L^p_k\) for Sobolev spaces of regularity \(k\),
in accordance with the literature in gauge theory 
and Floer theory in low dimensional topology.
These are often denoted by \(W^{k,p}\) or \(H^{k,p}\) or, 
when \(p=2\), simply by \(H^k\).

All manifolds are assumed to be smooth, 
submanifolds to be smoothly embedded,
and manifolds with boundary to have smooth boundary.

\paragraph{Acknowledgments.}
I would like to express gratitude to my graduate advisor Tom Mrowka
for his encouragement, support and sharing many insights
on geometry, topology and analysis.
I would like to thank Maciej Borodzik and Michał Miśkiewicz 
for helpful comments on a~draft of this paper,
and an~anonymous referee for pointing out an~innacuracy
in the statement of \namedref{thm:low-regularity}.

Most results of this paper have been part of the author's Ph.D. thesis
\cite{Suw2020}, although the proofs have been revised
and some results have been generalized.
In particular, \cite{Suw2020} does not consider \(4\)\hyp{}manifolds 
with more than two boundary components, 
or the general split Coulomb slice on them.

\section{Analytical preparation}
\label{sec:analytical-preparation}
We prove two results which are fundamental for the surjectivity proofs in
\autoref{sec:properties-of-moduli-spaces}.

The first one, the \namedref{thm:low-regularity}, is a~regularity theorem for
operators of the form \(D = D_0 + K : L^2_1 \to L^2\),
where \(D_0\) is a~first\hyp{}order elliptic operator
and \(K\) is a~compact operator.
In our applications \(K\) is a~multiplication 
by an~\(L^2_1\)\hyp{}configuration on
a~\(4\)\hyp{}manifold and thus it does not factor as a~map \(L^2_1 \to L^2_1\).
If it did, elliptic regularity would immediately imply
the regularity result.
The novelty here is the general form of the perturbation \(K\)
which is only assumed to be compact as a~map \(L^2_1 \to L^2\).

The second result, the \namedref{thm:unique-cont-dirac-Lr},
is a~strong unique continuation principle in a~similar low regularity setting.
This can be understood as a~strengthening of the unique continuation results
of \cite{KMbook}*{Section 7}.

\subsection{A~regularity theorem for Dirac operators}

Let \(X\) be a~smooth Riemannian manifold with 
asymptotically cylindrical ends (without boundary).
Let \(D_0\) be an~elliptic operator
\(D_0: C^\infty(X;E) \to C^\infty(X;F)\)
of order \(1\)
which is asymptotically cylindrical on the ends of \(X\).
Let \(K: L^2_1(X;E) \to L^2(X;F)\) be a~compact operator
which has a~compact formal adjoint
\(K^\ast: L^2_1(X;E) \to L^2(X;F) \)
and which extends to a~continuous operator
\(K : L^2(X; E) \to L^2_{-l}(X;F)\)
for some \(l\).
Consider the operator
\[D = D_0 + K : L^2_1(X;E) \to L^2(X;F). \]
The aim of this subsection is to prove the following regularity theorem:
\begin{regularitythm}
	\label{thm:low-regularity}
	Assume that \(v \in L^2(X;E) \) satisfies \(D v = h \)
	for some \(h \in L^2(X;F) \).
	Then \( v \in L^2_{1}(X;E) \).

	Moreover, if \(v, h \in L^2_{loc}(X;E)\) instead, 
	then \(v \in L^2_{1,loc}(X;E)\).
\end{regularitythm}

In the course of the proof
we will make use of the following simple lemma 
(cf. Lipyanskiy \cite{Lip2008}*{Lemma 44}):
\begin{lemma}
	\label{lem:op-weak-comp-convergence}
	Suppose the sequence of Hilbert space operators \( \{A_i : V \to W\} \) 
	is uniformly bounded and weakly convergent to \(A\) in the sense that
	for any \(v \in V\) we have \( A_i(v) \to A(v) \).

	If \( K : W \to U \) is compact, then 
	\( A_i \circ K \to A \circ K \).
\end{lemma}
We also need the following version of the G{\aa}rding inequality,
proven in \cite{Shu1992}*{Appendix 1, Lemma 1.4}:
\begin{proposition}[cylindrical G{\aa}rding inequality]
	\label{prop:cylindrical-Garding}
	For every \(s,t \in \mathbb{R}\) there exists \(C>0\)
	such that for any \(u \in C_0^\infty(X;E)\)
	\[ \lVert u\rVert_{L^2_{s+1}} \leq C( \lVert  D_0 u \rVert_{L^2_s} + \lVert u\rVert_{L^2_t}).\]
\end{proposition}

With these in hand, we are ready to prove
the \namedref{thm:low-regularity}.

\begin{proof}
	[Proof of \nameCref{thm:low-regularity}]
	By considering
	\[
		\begin{pmatrix}
			0 & D_0^\ast \\
			D_0 & 0
		\end{pmatrix}
		+
		\begin{pmatrix}
			0 & K^\ast \\
			K & 0
		\end{pmatrix}
		:L^2_{1}(X;E \oplus F) \to L^2(X; E \oplus F) 
	\]
	where \(D_0^\ast\) is the formal adjoint of \(D_0\)
	we may assume, without loss of generality, 
	that \(E = F\) and \(D_0\) is formally self\hyp{}adjoint.
	Since \(X\) has bounded geometry
	and \(D_0\) is uniformly elliptic,
	Proposition 4.1 in \cite{Shu1992}*{Section 1}
	implies that the minimal and maximal operators
	\(D_0^{min}\) and \(D_0^{max}\)
	of \(D_0: L^2(X;E) \to L^2(X;E)\) coincide,
	and their domains are equal to \(L^2_1(X;E)\).
	Thus \(D_0\) is essentially self-adjoint.

	Consider the equation
	\begin{equation}
		\label{eqn:dirac-large-c}
		(D_0+ic) \psi + K \psi = h + i c v
	\end{equation}
	for large \(c \in \mathbb{R}\).
	Certainly, \(\psi = v \in L^2(X;E)\) solves the equation.
	We will prove that for large \(c\) 
	equation \eqref{eqn:dirac-large-c} 
	has a~unique solution \(\psi\) in \(L^2_1(X;E)\),
	which is also unique in \(L^2(X;E)\),
	so that \(v = \psi \in L^2_1(X;E)\), as wished.

	Firstly, we want to prove \( D_0+ic : L^2_{1}(X;E) \to L^2(X;E) \)
	is invertible.
	Since the spectrum of \(D_0\) is real,
	\( (D_0+ic)^{-1}:L^2(X;E) \to L^2(X;E) \) exists 
	and is bounded by \( \frac{1}{|c|} \) 
	(cf. \cite{Lan1993}*{Chapter XIX, Theorem 2.4}).
	Therefore by
	\autoref{prop:cylindrical-Garding} we have the first inequality:
	\begin{align*}
		\lVert  (D_0+ic)^{-1} u \rVert_{L^2_1}
		&\leq \lVert  D_0 (D_0 + ic)^{-1} u\rVert_{L^2} + \lVert (D_0+ic)^{-1} u \rVert_{L^2}
		\\ & \leq
		\lVert  u - ic (D_0+ic)^{-1} u \rVert_{L^2} + (1/|c|) \lVert u\rVert_{L^2}
		\\ & \leq (2 + 1/|c|) \lVert u\rVert_{L^2} 
	\end{align*}
	and therefore \( (D_0+ic)^{-1} \)
	is a~bounded operator \(L^2(X;E) \to L^2_1(X;E)\).

	Secondly, we want to prove weak convergence of
	\( (D_0+ic)^{-1} \) to \(0\) as \(|c| \to \infty\),
	i.e., that for each \(v \in L^2(X;E)\)
	we have \( (D_0+ic)^{-1}(v) \to 0\) in \(L^2_1(X;E)\)
	as \(|c| \to \infty\).
	The spectral theorem for unbounded self-adjoint operators
	\cite{Lan1993}*{Chapter XIX, Theorem 2.7}
	provides us with an~orthogonal decomposition
	\( L^2(X;E) = \hat{\bigoplus}_n H_n\)
	such that the restriction
	\(D_n = D_0|_{H_n}:H_n \to H_n\)
	is a~bounded operator (in \(L^2\)\hyp{}norm on \(H_n\))
	and \(D_0 = \hat{\bigoplus}_n D_n\).
	In particular, \(H_n \subset \mathrm{Dom}(D_0) = L^2_1(X;E)\).
	For each \(v_n \in H_n\) we thus have,
	by \autoref{prop:cylindrical-Garding},
	\[ \lVert v_n\rVert_{L^2_1} \leq C( \lVert D_0 v_n\rVert_{L^2} + \lVert v_n\rVert_{L^2})
	\leq C (C_n + 1) \lVert v_n\rVert_{L^2} = C_n' \lVert v_n\rVert_{L^2}, \]
	where \(C_n = \lVert D_n\rVert_{L^2}\).
	Therefore
	\begin{equation*}
		\lVert  (D_0+ic)^{-1} v_n \rVert_{L^2_1}
		\leq C_n' \lVert  (D_0 + ic)^{-1} v_n \rVert_{L^2}
		\leq \frac{C_n'}{|c|} \lVert v_n\rVert_{L^2}
		\xrightarrow{|c| \to \infty} 0
	\end{equation*}
	which proves that for any \(v \in L^2(X;E)\),
	\( (D_0+ic)^{-1} v \xrightarrow{|c|\to \infty} 0 \) in \(L^2_1(X;E)\),
	i.e., \( (D_0 + ic)^{-1}\) converges weakly to \(0\),
	as wished.

	We proceed to proving existence and uniqueness of solutions to
	\eqref{eqn:dirac-large-c} in \(L^2_1(X;E)\) for large \(|c|\).
	Using \autoref{lem:op-weak-comp-convergence}
	we conclude that \( T = (D_0+ic)^{-1} K : L^2_1(X;E) \to L^2_1(X;E) \)
	converges strongly to \(0\).
	Thus, for large \(|c|\), 
	the operator \( \mathrm{Id} 
	+ (D_0+ic)^{-1} K : L^2_1(X;E) \to L^2_1(X;E)\) 
	is invertible.
	Composing with \((D_0 + ic)\) we conclude that
	\( D_0+ic + K = D+ic : L^2_1(X;E) \to L^2(X;E) \) is invertible
	for large \(|c|\).
	Existence and uniqueness of \( \psi \in L^2_1(X;E) \)
	solving \eqref{eqn:dirac-large-c} follows.

	To conclude the first part of the proof
	we need to establish uniqueness of solutions
	to \eqref{eqn:dirac-large-c} in \(L^2(X;E)\).
	We have that \( (D_0 + ic + K) (\psi - v) = 0\).
	This implies that \(\psi - v\) is perpendicular in \(L^2(X;E)\)
	to the image of \( D_0^\ast - ic + K^\ast = D_0 - ic + K
	:L^2_1(X;E) \to L^2(X;E)\).
	However, we already established that for large \(|c|\)
	the latter operator is invertible
	and it follows that \(\psi - v = 0\);
	thus \( v = \psi \in L^2_1(X;E) \).

	Finally, it remains to consider the more general case \(v, h \in L^2_{loc}\).
	For any compact
	\(A \subset X\) we can
	take a~compactly supported bump function \(\rho:X \to [0,1]\)
	with \(\rho|_A=1\)
	and define \(v' = \rho v \in L^2(X; E)\).
	Then \(D v' = h' \) for some \(h' \in L^2(X;E)\)
	and therefore \(v' \in L^2_1(X;E)\) as proven above,
	so \(v|_A \in L^2_1(A;E|_A)\).
	Repeating the argument for every compact \(A \subset X\)
	shows that \(v \in L^2_{1,loc}(X;E) \).
\end{proof}

\subsection{A strong UCP for Dirac operators}

Let \( M\) be a~connected Riemannian manifold
and \(S\) a~real (resp. complex) Dirac bundle over it
(e.g., the real (resp. complex) 
spinor bundle associated to a~\spinc{}-structure on \(M\))
with connection \(A\).
Denote by \(\Dirac_A:\Gamma(S) \to \Gamma(S)\) the corresponding Dirac operator.
Let \(V \in L^n(M;\mathbb{R})\) be a~potential.
Here we prove the following
unique continuation theorem
for spinors and potentials of low regularity.
\begin{strongucpthm}
	\label{thm:unique-cont-dirac-Lr}
	The differential inequality
	\begin{equation*}
		\mleft|\Dirac_A \Phi\mright| \leq V |\Phi|
	\end{equation*}
	has unique continuation property in \(L^{\frac{2n}{n+2}}_{1,loc}(M;S)\).
\end{strongucpthm}

The version of this theorem for \(d+d^\ast\)
instead of \(\Dirac_A\) has been proven in \cite{Wol1992}*{Theorem 2}:
\begin{theorem}
	[Wolff, \cite{Wol1992}*{Theorem 2}]
	\label{thm:wolff-ucp}
	Suppose \(M\) is an~\(n\)-dimensional manifold, \(n \geq 3\),
	\(p = \frac{2n}{n+2}\), 
	and \(\omega \in L^p_{1,loc}(\Omega^\ast(M))\)
	such that \( |d \omega| + |d^\ast \omega| \leq V |\omega|\)
	with \(V \in L^n_{loc}(M)\).
	Then if \(\omega\) vanishes on an~open set,
	it vanishes identically.
\end{theorem}
It thus suffices to reduce our problem to Wolff's result,
following the idea of \cite{Mand1994}.
Since the proof in \cite{Mand1994} does not explain
the reduction rigorously, we describe the procedure below.

\begin{proof}
	The problem is local, thus we need only to consider the case
	of \(M\) being \(\mathbb{R}^n\) with some metric \(g\).
	If \(A_0\) is the flat connection on \(S\),
	\(\Dirac_A - \Dirac_{A_0}\) is a~smooth operator of order \(0\),
	so (again using locality) we can assume that \(A\) is flat.
	Moreover, by contractibility of \(\mathbb{R}^n\),
	we can decompose \(S\) into irreducible components,
	and irreducible components must be isomorphic to the real (resp. complex)
	spinor bundle \(\widetilde S\) (resp. \(\widetilde S_{\mathbb{C}}\))
	associated to the unique spin structure.
	Thus we have reduced the problem to the case where 
	\(S = \widetilde S \otimes \mathbb{R}^k\)
	(resp. \(S = \widetilde S \otimes \mathbb{C}^k\)) for some \(k\).

	The real (resp. complex) spinor bundle \(\widetilde S\) (resp. \(\widetilde S_{\mathbb{C}}\))
	embeds into \(\mathrm{C\ell}_n(\mathbb{R}^n)\)
	(resp. \(\mathbb{C}\ell_n(\mathbb{R}^n)\)).
	Furthermore, \cite{LM1989}*{Theorem 5.12} implies
	that the Dirac operator on the Clifford bundle
	\(\mathrm{C\ell}_n(\mathbb{R}^n)\)
	(resp. \(\mathbb{C}\ell_n(\mathbb{R}^n) = \mathrm{C\ell}(\mathbb{R}^n)
	\otimes \mathbb{C}\))
	is equivalent to \( d + d^\ast\) on 
	\(\Lambda^\ast(\mathbb{R}^n)\)
	(resp. \(\Lambda^\ast(\mathbb{R}^n) \otimes \mathbb{C}\))
	via the canonical isomorphism 
	\(\mathrm{C\ell}(\mathbb{R}^n) \simeq \Lambda^\ast(\mathbb{R}^n)\).
	Thus, \(\Dirac_{A_0}\) is equivalent to
	\( (d+d^\ast) \otimes 1_{\mathbb{R}^k}
	\Lambda^\ast(\mathbb{R}^n) \otimes \mathbb{R}^k \to 
	\Lambda^\ast(\mathbb{R}^n) \otimes \mathbb{R}^k\)
	(resp. \( (d+d^\ast) \otimes 1_{\mathbb{C}^k}
	\Lambda^\ast(\mathbb{R}^n) \otimes \mathbb{C}^k \to 
	\Lambda^\ast(\mathbb{R}^n) \otimes \mathbb{C}^k\)).
	The proof of \autoref{thm:wolff-ucp} goes through
	for differential forms with coefficients in \(\mathbb{R}^k\)
	(resp. \(\mathbb{C}^k\)),
	establishing the unique continuation property for 
	\((d+d^\ast) \otimes \mathbb{R}^k\)
	(resp. \( (d+d^\ast) \otimes 1_{\mathbb{C}^k}\)).
	This finishes the proof.
\end{proof}

In the article we will use the following special case:
\begin{corollary}[UCP for Dirac operators in 4d]
	\label{cor:unique-cont-dirac-4d}
	In the setting of
	the \namedref{thm:unique-cont-dirac-Lr},
	assume \(n=4\) and \(V \in L^2_1(M;\mathrm{Aut}(S))\).
	Then any solution \(\Phi \in L^2_{1,loc}(M;S)\) to
	\begin{equation*}
		\Dirac_A \Phi + V \Phi = 0
	\end{equation*}
	which is zero on some open set
	is identically zero.
\end{corollary}

\section{Seiberg-Witten moduli spaces in split Coulomb slice}
\label{sec:Seiberg-Witten-moduli-spaces-in-split-Coulomb-slice}

In this section we introduce the moduli spaces of the Seiberg\hyp{}Witten equations
on a~Riemannian \(4\)\hyp{}manifold with boundary a~collection of rational homology
spheres,
together with the restriction maps to the boundary.
For \(3\)\hyp{}manifolds, we introduce the \emph{Coulomb slice}
and its \emph{polarization}, a~decomposition of the tangent space
into a~sum of two infinite\hyp{}dimensional subspaces.
For \(4\)\hyp{}manifolds, we introduce the \emph{double Coulomb slice}
and what we call the \emph{split Coulomb slice}
together with the \emph{split gauge group}.
Since the restriction maps are generally not invariant with respect
to the split gauge group, we need to introduce appropriate
\emph{twisted restriction maps} as well.

The split gauge fixing is a~key novel element that generalizes the gauge slice
introduced by Khandhawit \cite{Kha2015}
and Lipyanskiy \cite{Lip2008}.
It simplifies the proof of the \namedref{thm:composing-cobordisms}
letting us to reduce it to the case of untwisted restriction maps.

\subsection{Coulomb slice on \texorpdfstring{\(3\)}{3}-manifolds}

We begin by introducing
\emph{polarizations} on the Seiberg\hyp{}Witten 
configuration space in Coulomb gauge
on an~oriented rational homology sphere \(Y\)
and collections of such.

Let \(g\) be a~Riemannian metric
and \(\mathfrak{s}\) be a~\spinc{} structure on \(Y\).
Denote by \(S_Y\) the associated spinor bundle
and choose a~smooth \spinc{} connection \(B_0\).
The \emph{Seiberg\hyp{}Witten configuration space} on \(Y\)
is the space
\begin{equation*}
	\ConfThree{Y} = \ConfThreeFull{B_0}{Y}
\end{equation*}
consisting of pairs \((B,\Psi)\) of a~\emph{\spinc{} connection}
and a~\emph{spinor} on \(Y\).

In Seiberg\hyp{}Witten theory one investigates
the Chern\hyp{}Simons\hyp{}Dirac functional \(\mathcal{L}\)
\begin{equation*}
	\mathcal{L}(B, \Psi) 
	= - \frac{1}{8} \int_Y (B^t - B^t_0) 
	\wedge (F_{B^t} + F_{B^t_0} )
	+ \frac{1}{2} \int_Y \langle \Dirac_{B} \Psi, \Psi \rangle.
\end{equation*}
on the configuration space.
The \emph{gauge group} \(\Gauge(Y) = L^2_{3/2}(Y;S^1)\) acts on \(\ConfThree{Y}\)
via \(u (A, \Phi) = (A - u^{-1}du, u \Phi)\)
where \(u \in \Gauge(Y)\),
leaving \(\mathcal{L}\) invariant.
If one used spaces of higher regularity,
one could work with the quotient of the configuration space by the action
of the gauge group.
However, in the low regularity setting the action of \(\Gauge(Y)\)
on the spinors is not continuous.
Because of that 
(and in applications
concerning the Seiberg\hyp{}Witten stable homotopy type, cf. \cite{KLS2018}),
it is preferable to take the Coulomb slice as the model for the quotient
by the identity component of the gauge group.
Indeed, for \(Y\) a~rational homology sphere 
the Hodge decomposition gives
the \(L^2\)\hyp{}orthogonal decomposition
\begin{equation*}
	\Omega^1(Y) = \Omega^1_C(Y) \oplus \Omega^1_{cC}(Y)
\end{equation*}
where
\begin{math}
	\Omega^1_{cC}(Y) = \{ b \in \Omega^1(Y) | d^\ast b = 0\}
	= d^\ast (\Omega^2(Y))
\end{math}
is the space of (smooth) coclosed forms and
\begin{math}
	\Omega^1_C(Y) = \{ b \in \Omega^1(Y) | db = 0 \}
	= d(\Omega^0(Y))
\end{math}
is the space of (smooth) closed forms.
Denote by \(\proj_d\) the projection 
\begin{math}
	\proj_d: \SForms{1/2}{Y}
	\to \SForms[C]{1/2}{Y}
	= d\left( L^2_{3/2} (i \Omega^0(Y))\right)
\end{math}
along \(\SForms[cC]{1/2}{Y}\).
\begin{lemma}
	[gauge fixing in \(3\)d]
	\label{lem:gauge-fixing-in-3-d}
	On \(Y\), there is a~continuous choice of 
	\emph{based} and \emph{contractible} gauge transformations
	putting forms in the Coulomb slice, i.e., a~homomorphism
	\begin{align*}
		\SForms{1/2}{Y} &\to \GaugeIdBased(Y) \\
		a & \mapsto u_a
	\end{align*}
	such that
	\begin{math}
		a - u_a^{-1} du_a = \mleft( 1-\proj_d\mright)a \in \SForms[cC]{1/2}{Y}
	\end{math},
	where
	\begin{math}
		\GaugeIdBased(Y)
		= \left\{ e^f \middle| f \in L^2_{3/2}(i\Omega^0(Y)), \int_Y f = 0 \right\}
		\subset \Gauge(Y)
	\end{math}.
	For each \(a\), there is exactly one such \(u_a\).
\end{lemma}
\begin{proof}
	Denote \(\Omega^0_0(Y) = \left\{ f \in \Omega^0(Y) \middle| \int_Y f = 0 \right\} \).
	The exterior derivative 
	\(d : L^2_{3/2}(\Omega^0_0(Y)) \to \SForms[C]{1/2}{Y} \)
	has inverse \(G_d\).
	Denote by \(\Pi_d\) the orthogonal projection
	\(\Omega^1(Y) \to d(\Omega^0(Y))\).
	We take
	\(u_a = e^{G_d \Pi_d a}\)
	which has the desired properties.
	Uniqueness follows from the fact that
	\(df = 0\) and \(\int_Y f = 0\) imply \(f=0\).
\end{proof}
Moreover, for \(Y\) a~rational homology sphere
we have \(\Gauge(Y) = \GaugeId(Y) = 
\left\{ e^f \middle| f \in L^2_{3/2}(Y;i\mathbb{R}) \right\} \).
Indeed, taking \(\tilde{u} = u_{-u^{-1}du} u\)
gives us \(\tilde{u}\) with \(d(\tilde{u}^{-1} d \tilde{u})=0\),
which implies \(d \tilde{u}=0\) and thus \(\tilde{u}\)
is constant.
It follows that there is are bijections
\begin{align*}
	\CoulThree{Y} &\leftrightarrow
	\quotient{\mleft( \ConfThree{Y} \mright)}{\GaugeIdBased(Y)},
	\\
	\quotient{\mleft( \CoulThree{Y} \mright) }{S^1} &\leftrightarrow
	\quotient{\mleft( \ConfThree{Y} \mright)}{\Gauge(Y)},
\end{align*}
justifying the restriction to the Coulomb slice:
\begin{definition}
	[Coulomb slice]
	\label{def:Coulomb-slice}
	The \emph{Coulomb slice} on \(Y\) 
	with respect to the reference connection \(B_0\)
	is the space of configurations
	\begin{equation*}
		\CoulThree{Y}
		= \CoulThreeFull{B_0}{Y}
		\subset
		\ConfThree{Y}.
	\end{equation*}
\end{definition}

The following subspaces are crucial to the analysis of
Atiyah\hyp{}Patodi\hyp{}Singer boundary value problem
for the Seiberg\hyp{}Witten equations on \(4\)\hyp{}manifolds
with boundary \(Y\).
\begin{definition}
	[polarization on the Coulomb slice]
	\label{def:polarization-on-Coulomb-slice}
	We define
	\( H^+(Y,\mathfrak{s})\) (resp. \(H^-(Y,\mathfrak{s})\)) \(\subset \TangCoulThree{Y}\)
	to be the closure of the span of positive
	(resp. nonpositive)
	eigenvalues of 
	\begin{equation*}
		(\star d) \oplus \Dirac_{B_0} : 
		\TangCoulThree{Y} \to \TangCoulThree{Y}.
	\end{equation*}
	We denote by \(\Pi^\pm: \TangCoulThree{Y} \to H^\pm(Y,\mathfrak{s})\) 
	the projection onto
	\(H^\pm(Y,\mathfrak{s})\) along \(H^\mp(Y,\mathfrak{s})\).
\end{definition}
One of our goals is to prove that the moduli spaces
of solutions to the Seiberg\hyp{}Witten equations on \(X\)
are, in a~precise sense, comparable to 
the negative subspace \(H^-(Y,\mathfrak{s})\)
via the restriction map
(cf. \namedref{thm:semi-infinite-dimensionality-of-moduli-spaces}).

Finally, if \(Y = \bigsqcup_i Y_i\) is a~disjoint sum of 
oriented rational homology spheres \(Y_i\)
then we define
\begin{align*}
	\ConfThree{Y} 
	= \prod_i \ConfThree[i]{Y_i},
	&\quad
	\CoulThree{Y}
	= \prod_i \CoulThree[i]{Y_i},
	\\
	\mathcal{L}(\prod_i b_i,\prod_i \Psi_i)
	= \sum_i \mathcal{L}(b_i,\Psi_i)
	&\quad
	H^\pm(Y,\mathfrak{s})
	= \prod_i H^\pm(Y_i,\mathfrak{s}_i).
\end{align*}
Moreover, note that there are a~natural identifications
between configuration spaces for \(Y\) 
and oppositely oriented \(-Y\).
For a~\spinc{} structure \(\mathfrak{s}\)
with its spinor bundle \(S_Y\)
there is the conjugate \spinc{} structure \(\overline{\mathfrak{s}}\)
determined by the conjugate bundle \(\overline{S}_Y\),
and the anti\hyp{}linear isomorphism \(S_Y \simeq \overline{S}_Y\)
induces natural affine isomorphisms
\begin{align*}
	\ConfThree{Y} 
	&\simeq \Configuration{-Y}{\overline{\mathfrak{s}}}
	\quad
	\CoulThree{Y}
	\simeq \Configuration[cC]{-Y}{\overline{\mathfrak{s}}}
	\quad
	H^\pm(Y,\mathfrak{s})
	\simeq H^\mp(-Y,\overline{\mathfrak{s}}).
\end{align*}

\subsection{Split Coulomb slice on \texorpdfstring{\(4\)}{4}-manifolds}
\label{sec:split-Coulomb-slice-on-4-manifolds}

We turn our attention to the Seiberg\hyp{}Witten equations
and gauge fixings for configurations on a~connected oriented
\(4\)\hyp{}manifold \(X\) with nonempty boundary \(\partial X \neq
\varnothing\) satisfying \(b_1(\partial X) = 0\),
oriented using the outward normal.
As explained by Khandawit \cite{Kha2015},
the most convenient slice for these is a~kind of a~\textit{double} Coulomb slice
(which was already used by Lipyanskiy \cite{Lip2008}),
which imposes both coclosedness of the connection \(1\)\hyp{}form
on both \(X\) and \(\partial X\), as well as an~auxiliary condition
near \(\partial X\).
We drop this auxiliary condition from the definition of 
the \emph{double Coulomb slice}
and instead introduce the \emph{split Coulomb slice}
which generalizes the constructions of Khandhawit and Lipyanskiy.
This allows one to choose a~gauge fixing which do not require twisting
or ones that are more geometric in nature,
depending on one's needs.
Indeed, twisting is necessary in Khandhawit's and Lipyanskiy's gauge fixing,
in which
the restriction map may not commute with the residual gauge group action.

We begin by introducing the Seiberg\hyp{}Witten equations
on \(4\)\hyp{}manifolds.

\begin{definition}
	[Seiberg-Witten equations]
	\label{def:Seiberg-Witten-equations}
	The \emph{Seiberg\hyp{}Witten map} is defined by
	\begin{equation*}
		\SW : \ConfFour{X} \to \SWDomain{X},
	\end{equation*}
	\begin{equation*}
		\SW(A,\Phi) = \mleft( 
			\frac 1 2 F^+_{A^t} - \rho^{-1}((\Phi\Phi^\ast)_0), 
			\Dirac_{A}^+ \Phi 
		\mright),
	\end{equation*}
	where \(A^t\) denotes the connection induced by \(A\) on \(\det(S^+_X)\)
	and \(F^+_{A^t}\) denotes the self\hyp{}dual part of its curvature,
	according to the splitting \( \Lambda^2(X) = \Lambda^+(X) \oplus
	\Lambda^-(X)\) by the eigenspaces of the Hodge star \(\star\).

	The \emph{Seiberg\hyp{}Witten equations}
	are the equations given by \(\SW(A,\Phi)=0\), that is,
	\begin{equation*}
		\left\{
		\begin{aligned}
			\frac 1 2 F^+_{A^t} - \rho^{-1}((\Phi\Phi^\ast)_0) \hspace{-0.5em}
			&= 0, \\
			\Dirac_{A}^+ \Phi  &= 0.
		\end{aligned}
		\right.
	\end{equation*}
\end{definition}
Note that continuity and smoothness of the map \(\SW\) follow from
\autoref{thm:multiplication} and the
fact that continuous multilinear maps on Banach spaces
are smooth.

These equations are equivariant with respect to the action of the gauge group
\( \Gauge(X) = L^2_2(X;S^1) \).
As is easily seen, the solution set is invariant under this action.
\begin{lemma}
	[gauge group action on a~\(4\)\hyp{}manifold]
	\label{lem:gauge-action-4d}
	The gauge group \(\Gauge(X)\) acts on \(\ConfFour{X}\)
	via \(u (A, \Phi) = (A - u^{-1}du, u \Phi)\)
	where \(u \in \Gauge(X)\).
	Moreover, \(\SW(A, \Phi) = 0\) if and only if \(\SW(u(A,\Phi)) = 0\).
\end{lemma}
Note that this action is not continuous
since the multiplication \(L^2_2(X) \times L^2_1(X) \to L^2_1(X)\)
is not continuous.
It is well\hyp{}defined since 
\(\Gauge(X) \subset L^\infty(X) \cap L^2_2(X)\)
and the multiplication \( (L^2_2(X) \cap L^\infty(X)) 
\times L^2_1(X) \to L^2_1(X)\)
is continuous.

In order to prove that the moduli spaces of solutions are manifolds
we need to investigate the differential of \(\SW\).
\begin{align*}
	D_{(A,\Phi)}\SW 
	&: \TangFour{X} \to \SWDomain{X}, \\
	D_{(A,\Phi)}\SW&(a,\phi) 
	= (d^+ a, \Dirac^+_{A_0} \phi)
	+ ( -\rho^{-1}(\phi \Phi^\ast + \Phi \phi^\ast)_0,
	\rho(a) \Phi + \rho(\diff{A}) \phi).
\end{align*}
Similarly, at \( (e,A,\Phi) \) the differential of the gauge group action 
is:
\begin{align*}
	T \Gauge(X) = L^2_2(X; \mathbb{R}) \to& \TangFour{X},\\
	f \mapsto& (-df, f \Phi).
\end{align*}
As in dimension \(3\),
we can fix gauge using the Coulomb condition,
i.e., require that the \(1\)\hyp{}form is coclosed.
Adding the same condition on the boundary \(\partial X\)
ensures that the restriction to the boundary
lies in the previously defined Coulomb slice
(cf. \autoref{def:Coulomb-slice}):
\begin{definition}
	[double Coulomb slice]
	\label{def:double-Coulomb-slice}
	We define the \emph{double Coulomb slice}:
	\begin{equation*}
		\Omega^1_{CC}(X) = \{ a \in \Omega^1(X) | d^\ast a = 0,
		d^\ast(\iota^\ast_{\partial X} a) = 0 \}.
	\end{equation*}
	The gauge group preserving it is called the
	\emph{harmonic gauge group}:
	\begin{equation*}
		\GaugeHarm(X) = \{ u : X \to S^1 | u^{-1}du \in \SForms[CC]{1}{X}\}.
	\end{equation*}
\end{definition}
While the action of the full gauge group \(\Gauge(X)\) is not continuous,
the action of \(\GaugeHarm(X)\) is,
which will be proven in \autoref{lem:harmonic-gauge-group-acts-continuously}.

To define the split Coulomb slice we need to first understand
the harmonic gauge group and its relation to harmonic functions
and forms on \(X\).
Notice that we have a~well\hyp{}defined homomorphism
\begin{align}
	\label{eqn:gauge-diff}
	\delta : \Gauge(X) &\longrightarrow \SForms{1}{X} \\
	\nonumber
	u &\longmapsto \delta (u) = u^{-1}du
\end{align}
which, restricted to harmonic gauge transformations, induces
\begin{equation*}
	\delta : \GaugeHarm(X) \to i\mathcal{H}^1_D(X)
\end{equation*}
where 
\begin{equation*}
	\mathcal{H}^1_D(X) = \left\{ a \in \Omega^1(X) \middle|
	da=0, d^\ast a=0, \iota^\ast_{\partial X} a = 0 \right\} 
\end{equation*}
is the space of \emph{harmonic \(1\)\hyp{}forms with Dirichlet boundary conditions}.
Note that \(\delta\) (both on \(\Gauge(X)\) and on \(\GaugeHarm(X)\))
is an~inclusion modulo \(S^1\), i.e., \(\ker \delta = S^1\),
the group of constant gauge transformations.

On the other hand, the exponential map
\begin{align*}
	\exp : L^2_2(i \Omega^0(X)) 
	&\longrightarrow \Gauge(X)
	\\
	f & \longmapsto e^f
\end{align*}
restricted to the space of \emph{doubly harmonic functions}
\begin{equation*}
	\mathcal{H}(X)
	=
	\left\{ f \in \Omega^0(X) \middle| \Delta f = 0, 
	\Delta \mleft( f|_{\partial X} \mright) = 0 \right\}
\end{equation*}
yields a~homomorphism
\begin{equation*}
	\exp : i \mathcal{H}(X) \to \GaugeHarm(X)
\end{equation*}
since the conditions \(\Delta f = 0\) and \(\Delta(f|_{\partial X})=0\)
for an~imaginary\hyp{}valued function \(f\)
are equivalent to \(df \in \SForms[CC]{1}{X} \).
Importantly, the composition
\(\delta \circ \exp : L^2_2(i \Omega^0(X)) \to \SForms{1}{X}\)
is exactly the differential \(f \mapsto df\).

Denote the image of this exponential map by
\(\GaugeHarmId(X) = \exp( i \mathcal{H}(X))\).
As the next proposition explains,
\(\GaugeHarmId(X)\) is the identity component of \(\GaugeHarm(X)\).
Thus, our goal will be to find a~gauge fixing
that dispenses with the action of this identity component,
saving only the action by \(S^1\), the constant elements.
\begin{proposition}
	[sequence of harmonic gauge groups]
	\label{prop:sequence-of-harmonic-gauge-groups}
	The following sequence is exact:
	\begin{equation}
		\label{eqn:harmonic-gauge-sequence}
		\begin{tikzcd}
			0 \arrow{r} 
			& \GaugeHarmId(X) \arrow{r} 
			& \GaugeHarm(X) \arrow{r} 
			& \pi_0\GaugeHarm(X) \arrow{r} 
			& 0.
		\end{tikzcd}
	\end{equation}

	The identity component
	\(\GaugeHarmId(X)\) is isomorphic to 
	\(S^1 \times \mathbb{R}^{b_0(\partial X)-1}\)
	and the group of components 
	\(\pi_0\GaugeHarm(X)\)
	is naturally isomorphic to \(H^1(X; \mathbb{Z})\).
\end{proposition}
\begin{remark}
	Recall that \(H^1(X; \mathbb{Z}) \simeq \mathrm{Hom}(\pi_1(X), \mathbb{Z})\) 
	has no torsion.
\end{remark}

\begin{proof}
	Crucial to the understanding of the gauge group is the homomorphism
	\eqref{eqn:gauge-diff}.
	Hodge theory provides an~identification
	\begin{math}
		\mathcal{H}^1_D(X) \simeq H^1(X,\partial X; \mathbb{R}).
	\end{math}
	Thus, we will consider \(\delta\) as a~map
	\begin{math}
		\delta : \GaugeHarm(X) \to H^1(X, \partial X; i \mathbb{R})
	\end{math}
	with kernel \(S^1\).
	This will be used to establish a~map of horizontal short exact sequences
	\begin{equation}
		\label{eqn:gauge-groups-and-homology}
		\begin{tikzcd}
			&[-20pt] S^1 \arrow[hookrightarrow]{d}{} \arrow[equals]{r}{}
			&[-32pt] S^1 \arrow[hookrightarrow]{d}{}
			&[-20pt] 0 \arrow{d}{}
			&[-42pt] \\
			0 \arrow{r}{} 
			& \GaugeHarmId(X) \arrow[twoheadrightarrow]{d}{} \arrow[hookrightarrow]{r}{}
			& \GaugeHarm(X) \arrow{d}{\frac{\delta}{2 \pi i}} \arrow{r}{}
			& \pi_0 \GaugeHarm(X) 
			\arrow[hookrightarrow]{d}{\frac{[\delta]}{2 \pi i}} \arrow{r}{}
			& 0
			\\
			0 \arrow{r}{}
			& \quotient{H^0(\partial X; \mathbb{R})}{H^0(X; \mathbb{R})}
			\arrow{r}{} \arrow{d}{}
			& H^1(X, \partial X; \mathbb{R}) \arrow{r}{}
			\arrow[twoheadrightarrow]{d}{}
			& H^1(X; \mathbb{R}) \arrow{r}{} \arrow[twoheadrightarrow]{d}{}
			& 0
			\\
			& 0
			& \quotient{H^1(X; \mathbb{R})}{H^1(X; \mathbb{Z})}
			\arrow[equals]{r}{}
			& \quotient{H^1(X; \mathbb{R})}{H^1(X; \mathbb{Z})}
			& 
		\end{tikzcd}
	\end{equation}
	where the vertical sequences are also exact,
	as will be shown in the course of the proof.

	Firstly, we prove that \( \pi_0 \GaugeHarm(X) \simeq H^1(X; \mathbb{Z})\).
	Notice that for any closed loop 
	\(\gamma \subset X\) the period of \(\delta(u)\),
	i.e., the integral \(\int_\gamma \delta(u)\),
	is an~integer multiple of \(2 \pi i\);
	and it is zero whenever \(\gamma\) is contractible.
	(In fact it is the obstruction to lifting
	\(u|_\gamma : \gamma \to S^1\) to a~map
	\(u|_\gamma : \gamma \to \mathbb{R}\).)
	This way any \(u \in \GaugeHarm(X)\)
	determines an~element \( [u] \in H^1(X; 2 \pi i \mathbb{Z})\)
	and we get a~homomorphism
	\(\GaugeHarm(X) \to H^1(X; 2 \pi i \mathbb{Z})\).
	Since the periods (having values in \(2 \pi i \mathbb{Z}\)) 
	do not change under homotopy,
	this descends to a~map
	\(\pi_0 \GaugeHarm(X) \to H^1(X; 2 \pi i \mathbb{Z})\)
	and from its construction it follows
	that it coincides with the composition
	\begin{equation*}
		\GaugeHarm(X) \xrightarrow{\delta} 
		H^1(X, \partial X; i \mathbb{R})
		\to H^1(X; i \mathbb{R})
	\end{equation*}
	which has image in \(H^1(X; 2 \pi i \mathbb{Z})\).
	It remains to notice that any element 
	in \(x \in H^1(X; 2 \pi i \mathbb{Z})\)
	can be lifted to an~element 
	\(\widetilde{x} = H^1(X, \partial X; i \mathbb{R}) \simeq i
	\mathcal{H}^1_D(X)\)
	and then integrated along curves 
	to obtain an~element \(u_{\tilde{x}} \in \GaugeHarm(X)\)
	mapping to \(x\) (see \eqref{eqn:integrating-1-forms}).
	After dividing \(1\)\hyp{}forms by \(2\pi i\) 
	we obtain a~natural isomorphism
	\(\pi_0 \GaugeHarm(X) \simeq H^1(X; \mathbb{Z})\),
	as wished.

	Moreover, this shows that the kernel \(K\) of the map
	\( \GaugeHarm(X) \to H^1(X; 2 \pi i \mathbb{Z})\)
	maps via \(\delta\) to the kernel of
	\(H^1(X, \partial X; i \mathbb{R}) \to H^1(X; i \mathbb{R})\).
	We thus get the map
	\begin{equation*}
		\delta|_K : K \to \im(H^0(\partial X; i \mathbb{R}))
	\end{equation*}
	to the image of \(H^0( \partial X; i \mathbb{R})\) in
	\(H^1(X, \partial X; i \mathbb{R})\).
	The map \(\delta|_K\) itself has kernel \(S^1\).

	With this in mind, we turn our focus to \(\GaugeHarmId(X) =
	\exp(\mathcal{H}(X))\).
	Since any harmonic function \(f\) on \(X\) is determined
	by its restriction \(f|_{\partial X}\) to \(\partial X\),
	and harmonic functions on \(\partial X\) are locally constant,
	we have an~isomorphism
	\begin{equation*}
		H^0(\partial X; \mathbb{R})
		\simeq
		\mathcal{H}(X)
	\end{equation*}
	and we will denote any element \(g\) in \(H^0(\partial X; \mathbb{R})\)
	by \(f|_{\partial X}\) for the unique \(f \in \mathcal{H}(X)\)
	such that \(f|_{\partial X} = g\).
	We obtain a~canonically defined surjection
	\begin{align*}
		\exp : H^0(\partial X; i\mathbb{R}) &\longrightarrow \GaugeHarmId(X) \\
		f|_{\partial X} &\longmapsto e^{f}
	\end{align*}
	with kernel generated by restrictions of \(H^0(X;2\pi i \mathbb{Z})\).
	After dividing by \(2 \pi i\) we get an~isomorphism
	\(\GaugeHarmId(X) \simeq 
	H^0(\partial X; \mathbb{R})/H^0(X;\mathbb{Z})
	\simeq
	S^1 \times 
	(H^0(\partial X; \mathbb{R}) / H^0(X;\mathbb{R}))\), as wished.

	Finally, recall that the composition
	\begin{equation*}
		H^0(\partial X; i \mathbb{R})
		\simeq \mathcal{H}(X)
		\xrightarrow{\exp} 
		\GaugeHarmId(X) \xrightarrow{\delta} 
		H^1(X, \partial X; i\mathbb{R})
	\end{equation*}
	is given by \( f|_{\partial X} \mapsto df \)
	and therefore, by Hodge theory,
	represents the boundary map in the exact sequence
	\begin{equation*}
		0 \to H^0(X;\mathbb{R}) \to H^0(\partial X; \mathbb{R})
		\to H^1(X,\partial X; \mathbb{R})
		\to H^1(X; \mathbb{R})
		= 0
	\end{equation*}
	Thus \(\delta(\GaugeHarmId(X)) = \im(H^0(\partial X; i\mathbb{R}))
	= \delta(K)\).
	Since \(\ker \delta|_{\GaugeHarmId(X)} = S^1
	= \ker \delta|_{K}\)
	and \(\GaugeHarmId(X) \subseteq K\),
	we obtain that \(\GaugeHarmId(X) = K\),
	i.e., the sequence \eqref{eqn:harmonic-gauge-sequence} is exact,
	as wished.
\end{proof}

\begin{lemma}
	[\(\GaugeHarm\) acts continuously]
	\label{lem:harmonic-gauge-group-acts-continuously}
	The action of \(\GaugeHarm(X)\) on \(\ConfFour{X}\) is smooth.
\end{lemma}
\begin{proof}
	By \autoref{prop:sequence-of-harmonic-gauge-groups}
	it suffices prove that the action of
	\(\GaugeHarmId(X)\) is smooth.
	Further, by \autoref{thm:multiplication}
	it suffices to prove that \(\GaugeHarmId(X) \subset L^2_3(X; S^1)\)
	and that this injection is continuous.
	Since \(f \in i \mathcal{H}(X)\) and \(\mathcal{H}(X)\) 
	is finite\hyp{}dimensional,
	there are constants \(C, C'\)
	such that
	\(\|\exp(f)\|_{L^2_3} \leq C \|f\|_{L^2_3} \leq CC' \|f\|_{L^2_2}\),
	finishing the proof.
\end{proof}

Let us compare different splittings.
If \(s, s'\) are two different splittings,
then for any \([u] \in \pi_0 \GaugeHarm(X)\) we have 
\(s (s')^{-1} \in \GaugeHarmId(X)\).
Therefore any two gauge splittings differ by a~homomorphism
\(\pi_0 \GaugeHarm(X) \to \GaugeHarmId(X)\).

Our goal is to reduce the gauge group action
to the action of \(S^1\) and the action of a~chosen
lift of \(\pi_0\GaugeHarm(X)\) to \(\GaugeHarm(X)\).
Precisely, we will consider splittings
\(s : \pi_0\GaugeHarm(X) \to \GaugeHarm(X)\)
of \eqref{eqn:harmonic-gauge-sequence}.
In order to choose the gauge fixing we need to understand
that a~gauge splitting induces another splitting on the level of homology.
\begin{proposition}
	[homological splitting]
	\label{prop:homological-splitting}
	Any splitting \(s : \pi_0\GaugeHarm(X) \to \GaugeHarm(X)\)
	of \eqref{eqn:harmonic-gauge-sequence}
	induces a~splitting 
	\(s^H : H^1(X;\mathbb{R}) \to H^1(X,\partial X;\mathbb{R})\)
	of the exact sequence
	\begin{equation}
		0 \to
		H^0(\partial X; \mathbb{R}) / H^0(X;\mathbb{R})
		\to
		H^1(X,\partial X;\mathbb{R})
		\to
		H^1(X;\mathbb{R})
		\to
		0.
		\label{eqn:relative-homology-exact-sequence}
	\end{equation}
\end{proposition}
\begin{proof}
	Composing
	\(H^1(X; \mathbb{Z}) \xrightarrow{\cdot 2 \pi i} 
	\pi_0 \GaugeHarm(X) \xrightarrow{s} \GaugeHarm(X)
	\xrightarrow{\frac{\delta}{2 \pi i}} H^1(X, \partial X; \mathbb{R})\)
	we get a~linear map which by linearity uniquely extends to a~section
	\(s_H: H^1(X; \mathbb{R}) \to H^1(X, \partial X; \mathbb{R})\)
	of the aforementioned exact sequence.
\end{proof}

\begin{definition}
	[gauge splitting]
	\label{def:gauge-splitting}
	By a~\emph{gauge splitting} we call a~splitting
	\({s : \pi_0\GaugeHarm(X) \to \GaugeHarm(X)}\)
	of the exact sequence \eqref{eqn:harmonic-gauge-sequence}.
	We denote by
	\({s^H : H^1(X;\mathbb{R}) \to H^1(X,\partial X;\mathbb{R})}\)
	the associated \emph{homological splitting}.
\end{definition}

We clarify the relationship between gauge splittings and homological splittings.
\begin{proposition}
	[gauge splittings from homological splittings]
	\label{prop:gauge-splittings-from-homological-splittings}
	\hfill \\
	Let \(\sigma : H^1(X; \mathbb{R}) \to H^1(X, \partial X; \mathbb{R})\)
	be a~homological splitting, i.e., a~splitting of
	\eqref{eqn:relative-homology-exact-sequence}.
	Then up to action of \(S^1\) there exists a~unique gauge splitting \(s\)
	such that \(\sigma = s^H\).
\end{proposition}
\begin{proof}
	For existence, choose \(x_0 \in X\)
	and consider the map
	\begin{align}
		\label{eqn:integrating-1-forms}
		I_{x_0} : \left\{ \eta \in \mathcal{H}^1_D(X)\right.%
		&\left|\, [\eta] \in H^1(X; 2 \pi i \mathbb{Z}) \right\}
		\longrightarrow \GaugeHarm(X),
		\\
		\nonumber
		\eta &\longmapsto \left( I_{x_0}(\eta)(x) = 
		\exp\left(\int_{x_0}^x \eta \right) \right).
	\end{align}
	Notice that \(\delta(I_{x_0}(\eta)) = \eta\).
	Therefore taking \(s([u]) = I_{x_0}(\sigma([\delta(u)]))\)
	we get a~gauge splitting \(s: \pi_0 \GaugeHarm(X) \to \GaugeHarm(X)\)
	with \(\sigma = s^H\).

	The uniqueness up to action of \(S^1\) follows from
	the exactness of the middle vertical sequence in 
	\eqref{eqn:gauge-groups-and-homology}.
\end{proof}

In order to find the appropriate gauge fixing
we need the following analogue of \autoref{prop:sequence-of-harmonic-gauge-groups}
for \(1\)\hyp{}forms.
\begin{lemma}
	[decomposing \(1\)\hyp{}forms]
	\label{lem:decomposing-1-forms}
	\hfill \\
	\(\Omega^1(X) = \Omega^1_{CC}(X) + d(\Omega^0(X))\)
	and \(\Omega^1_{CC}(X) \cap d(\Omega^0(X)) = d(\mathcal{H}(X))\).
\end{lemma}
\begin{proof}
	This follows from the proof of \cite{Kha2015}*{Proposition 2.2}.
	(Note that our definition of \(\Omega^1_{CC}(X)\) 
	differs from Khandhawit's, 
	which we denote by \(\Omega^1_{s^\perp}(X)\) (cf.
	\autoref{def:orthogonal-splitting}).)
\end{proof}

In particular, we can decompose \(1\)\hyp{}forms as
\begin{equation*}
	\Omega^1(X) = \Omega^1_{CC}(X) \oplus d(\Omega^0_\partial(X))
\end{equation*}
where
\begin{equation*}
	\Omega^0_\partial(X) = 
	\left\{ f \in \Omega^0(X) \middle| \int_{Y_i} f = 0 \text{ for each
	component }Y_i \subset \partial X \right\}.
\end{equation*}
Denoting by \(\proj_{CC}\) the projection onto 
\(\Omega^1_{CC}(X)\)
along 
\(d(\Omega^0_\partial(X))\) 
we obtain the following analog of \autoref{lem:gauge-fixing-in-3-d}.
\begin{lemma}
	[Coulomb gauge fixing in \(4\)d]
	\label{lem:Coulomb-slice-fixing-in-4d}
	\hfill\\
	There is a~unique homomorphism
	\begin{align*}
		\SForms{1}{X} 
		& \to \Gauge^{e,\partial}(X) 
		= \exp\mleft( L^2_2\mleft(i \Omega^0_\partial(X)\mright)\mright)
		\\
		a & \mapsto u^{CC}_a
	\end{align*}
	such that 
	\begin{equation*}
		a - \left(u^{CC}_a\right)^{-1} d u^{CC}_a = \Pi_{CC} a 
		\in \SForms[CC]{1}{X}.
	\end{equation*}
\end{lemma}
\begin{proof}
	The projection \( (1-\Pi_{CC})\) on \(\Omega^1(X)\)
	has image in \(d( \Omega^0_\partial(X))\)
	and that \(d\) is injective on \(\Omega^0_\partial(X)\).
	Therefore there is a~unique homomorphism
	\(L^2_1(i \Omega^1(X)) \to L^2_2(i \Omega^0_\partial(X))\)
	sending \(a\) to the unique \(f_a^{CC} \in \Omega^0_\partial(X)\)
	such that \(a - df_a^{CC} \in \Omega^1(X)\).
	Then we take \(u_a^{CC} = \exp(f_a^{CC})\).
\end{proof}

We can further decompose
\begin{equation}
	\label{eqn:decomposing-CCs}
	\Omega^1_{CC}(X) = 
	\mleft(\Omega^1_{CC}(X) \cap (\mathcal{H}^1_D(X))^\perp\mright)
	\oplus \mathcal{H}^1_D(X).
\end{equation}
A~homological splitting \(s^H\) provides a~decomposition
\begin{equation}
	\label{eqn:decomposing-harmonics}
	\mathcal{H}^1_D(X) = s^H(H^1(X;\mathbb{R})) \oplus d(\mathcal{H}(X))
\end{equation}
which is an~analogue of 
\eqref{eqn:harmonic-gauge-sequence}
for \(\mathcal{H}^1_D(X)\).
With these in hand, we are ready to define the split Coulomb slice.
\begin{definition}
	[split Coulomb slice]
	\label{def:split-Coulomb-slice}
	Let \(s\) be a~gauge splitting and \(s^H\) its associated
	homological splitting.
	The \emph{split Coulomb slice} is
	\begin{equation*}
		\Omega^1_{s}(X) = \{ a \in \Omega^1_{CC}(X) |
			a \in (\Omega^1_{CC}(X) \cap (\mathcal{H}^1_D(X))^\perp)
			\oplus s^H(H^1(X;i\mathbb{R})) \}.
	\end{equation*}
\end{definition}
In particular, we have that
\begin{equation}
	\label{eqn:CC-exact-decomposition}
	\Omega^1(X) = \Omega^1_{s}(X) \oplus d(\Omega^0(X))
\end{equation}
and parallel to \autoref{lem:Coulomb-slice-fixing-in-4d}
and \autoref{lem:gauge-fixing-in-3-d}
we can use the projection \(\proj_s\) onto the first factor 
along the second one to obtain:
\begin{lemma}
	[split gauge fixing in 4d]
	\label{lem:split-gauge-fixing-in-4d}
	There is a~unique homomorphism
	\begin{align*}
		\SForms{1}{X} 
		& \to \GaugeIdBased(X) \\
		a & \mapsto u^s_a
	\end{align*}
	such that 
	\begin{equation*}
		a - \left(u^s_a\right)^{-1} d u^s_a = \proj_s a 
		\in \SForms[s]{1}{X}.
	\end{equation*}
\end{lemma}
\begin{proof}
	The projection \( (1-\Pi_s)\) on \(\Omega^1(X)\)
	has image in \(d( \Omega^0(X))\)
	and \(d\) is injective on \(\Omega^0_0(X)\).
	Therefore there is a~unique homomorphism
	\(L^2_1(i \Omega^1(X)) \to L^2_2(i \Omega^0_0(X))\)
	sending \(a\) to the unique \(f_a^{s} \in \Omega^0_0(X)\)
	such that \(a - df_a^s \in \Omega^1(X)\).
	Then we take \(u_a^s = \exp(f_a^s)\).
\end{proof}
\begin{remark}
	[continuous gauge fixing within double Coulomb slice]
	\label{rmk:continuous-gauge-fixing-within-double-Coulomb-slice}
	If we only consider \(a \in \Omega^1_{CC}(X)\), then
	the above map has image in \(\GaugeHarm(X)\), which is finite\hyp{}dimensional.
	Since the latter gauge group acts continuously on the configuration space
	by \autoref{lem:harmonic-gauge-group-acts-continuously},
	we conclude that putting \( (A,\Phi) \in \CoulFour{CC}{X}\)
	into split Coulomb slice \(\CoulFour{s}{X}\) can be done continuously
	with respect to \((A,\Phi)\).
\end{remark}

The gauge group acting on this split Coulomb slice is
the product of \(S^1\) and the split gauge group:
\begin{definition}
	[split harmonic gauge group]
	\label{def:gauge-group-split}
	Let \(s\) be a~gauge splitting.
	The \emph{split gauge group} is defined to be
	\begin{equation*}
		\GaugeSplit[s](X) = s(\pi_0\GaugeHarm(X)).
	\end{equation*}
\end{definition}
\begin{lemma}
	[split gauge group preserves the split Coulomb slice]
	\label{lem:split-gauge-compatible}
	For \(u \in \GaugeSplit[s](X)\)
	we have \(u^{-1}du \in \SForms[s]{1}{X}\).

	Conversely, if \(u^{-1} du \in \SForms[s]{1}{X}\),
	then for some \(z \in S^1\) we have
	\(zu \in \GaugeSplit[s](X)\).
\end{lemma}
\begin{proof}
	One direction follows directly from the definition of \(s^H\):
	if \(u = s([u])\), then \(s^H([\delta(u)]) = \delta(u) = u^{-1}du
	\in \mathcal{H}^1_D(X) \simeq H^1(X,\partial X;i \mathbb{R})\),
	so \(u^{-1}du \in \im s^H \subset \SForms[s]{1}{X}\).

	The other direction follows by chasing arrows in the diagram
	\eqref{eqn:gauge-groups-and-homology}.
\end{proof}

The circle \(\circ\) in the superscript indicates that
the only constant gauge transformation contained
in \(\GaugeSplit[s](X)\) is the identity.
This way we do not forget the \(S^1\)\hyp{}action
when taking the quotient by the split gauge group.

We want to compare different split slices
together with the split gauge group actions.
Choose two splittings \(s, s'\).
These determine a~map 
\( s' \cdot s^{-1} : \pi_0 \GaugeHarm(X) \to \GaugeHarmId(X)\).
Viewing \(\pi_0 \GaugeHarm(X)\) as a~sublattice
\(\pi_0 \GaugeHarm(X) \simeq H^1(X; 2 \pi i \mathbb{Z}) \subset H^1(X; i \mathbb{R}),\)
let \(\nu : H^1(X; i \mathbb{R}) \to \GaugeHarmId(X)\) 
be any homomorphism extending \(s' \cdot s^{-1}\).
Define
\begin{align*}
	F_\nu : \CoulFour{s'}{X} &\longrightarrow \CoulFour{s}{X} \\
	(A, \Phi) & \longmapsto \nu\left( \proj_{\im s^H}(A-A_0) \right) \cdot (A, \Phi).
\end{align*}
where \(\proj_{\im s^H} : i\Omega^1_{CC}(X) \to s^H(H^1(X; i\mathbb{R}))\)
is the projection
along \( i (\Omega^1_{CC}(X) \cap \mathcal{H}^1_D(X)) \oplus i d(\mathcal{H}(X))\)
(cf. \eqref{eqn:decomposing-CCs}, \eqref{eqn:decomposing-harmonics}).
\begin{proposition}
	[equivalence of slices]
	\label{prop:equivalence-of-slices}
	The map \(F_\nu\) is well\hyp{}defined,
	a~diffeomorphism,
	equivariant with respect to the action of
	\(\pi_0 \GaugeHarm(X) \simeq \GaugeSplit[s'](X) \simeq \GaugeSplit[s](X)\).
\end{proposition}
\begin{proof}
	Firstly, we need to show that the image of \(F_\nu\)
	actually lies in \(\CoulFour{s}{X}\).
	Equivalently, we want to show that
	\begin{align*}
		A_\nu : \Omega^1_{CC}(X) & \longrightarrow \Omega^1_{CC}(X) \\
		a &\longmapsto \nu\mleft(\proj_{\im s^H} a \mright) \cdot a
		= a + \delta\mleft(\nu\mleft(\proj_{\im s^H} a \mright) \mright)
	\end{align*}
	maps \(\Omega^1_s(X)\) to \(\Omega^1_{s'}(X)\).
	We have \(\delta(\GaugeHarmId(X)) = d(\mathcal{H}(X))\)
	and, moreover, \(\delta \circ \nu\) is a~homomorphism,
	thus a~linear map \(H^1(X; i\mathbb{R}) \to d(\mathcal{H}(X))\).
	What follows is that \(A_\nu\) is a~linear map.
	Now \(A_\nu\) restricted to 
	\(\Omega^1_{CC}\mleft(X\mright) \cap (\mathcal{H}^1_D(X))^\perp\)
	is identity by definition,
	so it suffices to show \(A_\nu( \im s^H ) \subset \im ((s')^H)\).
	Furthermore, \(H^1(X; i \mathbb{R})\) is spanned by
	\([\delta(\pi_0\GaugeHarm(X))] \) and therefore 
	\(s^H(H^1(X; i\mathbb{R}))\) is spanned by 
	\(\delta(\GaugeSplit[s](X))\),
	so it suffices to show \(A_\nu(\delta(\GaugeSplit[s](X))) \subset
	\im((s')^H)\).
	But we defined \(\nu\) so that for any \(u \in \GaugeSplit[s](X)\)
	\(\nu(\delta(u)) \cdot u \in \GaugeSplit[s'](X) \),
	which implies \(A_\nu( \delta(u)) = \delta( \nu(\delta(u)) \cdot u)
	\in \im( (s')^H)\), as wished.

	The map \(F_\nu\) is smooth because the map \(\nu\) is smooth
	and the action of the finite\hyp{}dimensional \(\GaugeHarm(X)\)
	on \(\ConfFour{X}\) is smooth.

	It is invertible because \(F_{1/\nu}\) 
	is its inverse.
	Indeed, since \(\im \nu \subset \GaugeHarmId(X)\),
	we have that \(\delta \nu \in d(\mathcal{H}(X))\),
	so \(\proj_{\im s^H}(\delta \nu) \equiv 0\).
	Therefore 
	\begin{equation*}
		\proj_{\im s^H} \left(
		\nu(\proj_{\im s^H}(A-A_0)) A - A_0\right)
		= \proj_{\im s^H}(A-A_0),
	\end{equation*}
	so
	\begin{equation*}
		\left(\nu\left(\proj_{\im s^H} \left(
		\nu(\proj_{\im s^H}(A-A_0)) A - A_0\right)\right)\right)^{-1}
		= (\nu(\proj_{\im s^H}(A-A_0)))^{-1}
	\end{equation*}
	and \(F_{1/\nu} \circ F_\nu = \mathrm{id}\) follows.
\end{proof}

Finally, we discuss the gauge slice used by Lipyanskiy \cite{Lip2008} and
Khandhawit \cites{Kha2015,KLS2018}.
They require that \(a \in \Omega^1_{CC}(X)\) 
and that for each component
\(Y_i \subset \partial X\) we have 
\(\int_{Y_i} \iota^\ast(\star a) = 0\).
Using Stokes' theorem one can show 
that for \(a \in \Omega^1_{CC}(X)\) this integral condition 
is equivalent to the condition that
\(\int_X df \wedge \star a = 0\) for any \(f \in \mathcal{H}(X)\).
This fits into our setup perfectly, since there is exactly one
homological splitting \(s^H_\perp\) such that 
\(\im s^H_\perp = \mathcal{H}^1_D(X) \cap (d(\mathcal{H}(X)))^\perp\).
\begin{definition}
	[orthogonal splitting]
	\label{def:orthogonal-splitting}
	We call \(s^H_\perp\) the \emph{orthogonal homological splitting}.
	We say that a~splitting \(s\) is
	a~\emph{orthogonal splitting}
	if \(s^H = s^H_\perp\).
\end{definition}

\subsection{Restriction to the boundary and twisting}
\label{sec:restriction-to-the-boundary-and-twisting}

Unless \(\partial X\) is connected,
we are not guaranteed that the \emph{restriction} to the boundary
is invariant under the action of the split gauge group
\(\GaugeSplit[s](X)\).
If it happens to be invariant for some \(s\), 
we call such \(s\) an~\emph{integral splitting}.
For a~general \(s\), we introduce and prove the existence of \emph{twistings}
of the restriction map, making it invariant under
the action of \(\GaugeSplit[s](X)\) action even for non\hyp{}integral \(s\).
As mentioned before, integral splittings are utilized in the proof of
the \namedref{thm:composing-cobordisms},
while non\hyp{}integral splittings may be more convenient in other contexts
(e.g., in constructions of \cites{Lip2008,Kha2015,KLS2018}).

We start by defining the restriction maps
for an~embedding \(\iota_Y : Y \hookrightarrow X\)
of an~oriented \(3\)\hyp{}manifold \(Y\).
Denote by \(\mathfrak{s}\) the restriction to \(Y\)
of the \spinc{} structure \(\hat{\mathfrak{s}}\) on \(X\).
We get canonical identifications \(S^\pm_X|_Y \simeq S_Y\).
Assuming \(Y\) is a~geodesic codimension\hyp{}\(1\) submanifold of \(X\),
the \spinc{} connection \(A_0\) induces a~\spinc{} connection \(B_0\) on \(Y\)
by simple restriction: 
\(\nabla^{B_0} = \iota_Y^\ast \nabla^{A_0}\).
Let \(a \in \SForms{1}{X}\), 
\(A \in \SConnections{1}{A_0}{X}\),
\(\Phi \in \SSpinors[+]{1}{X}\)
and \(u \in \Gauge(X)\).
We define the restrictions:
\begin{align*}
	R(a) &= \iota_Y^\ast(a) \in \SForms{1/2}{Y}, \\
	R(A) &= B_0 + \iota_Y^\ast(A-A_0) \in \SConnections{1/2}{B_0}{Y}, \\
	R(\Phi) &= \Phi|_Y \in \SSpinors{1/2}{Y}, \\
	R(u) &= u|_Y \in \Gauge(Y).
\end{align*}

Integral splittings are the ones for which restriction maps
are invariant under the split gauge group.
\begin{definition}
	[integral splitting]
	\label{def:integral-splitting}
	We call a~gauge splitting \(s\) \emph{integral}
	if for each \(u \in \GaugeSplit[s](X)\)
	we have \(u|_{\partial X} \equiv 1\).

	Equivalently, \(s\) is integral if the composition
	\(\pi_0\GaugeHarm(X) \xrightarrow{s} 
	\GaugeHarm(X) \xrightarrow{R} 
	\GaugeHarm(\partial X) \simeq (S^1)^{\pi_0(\partial X)}\)
	is trivial.
\end{definition}

The integrality of \(s\) is closely connected to the integrality of \(s^H\).
\begin{proposition}
	[homological classification of integral splittings]
	\label{prop:homological-classification-of-integral-splittings}
	If \(s\) is integral, then 
	\(s^H(H^1(X;\mathbb{Z})) \subset H^1(X, \partial X; \mathbb{Z})\),
	i.e., 
	\(s^H\) is integral as well.
	
	Given any integral homological splitting \(\sigma\)
	there exists a~unique integral splitting
	\(s\) such that \(\sigma = s^H\).
\end{proposition}
\begin{proof}
	Assume \(s\) is integral.
	Choose \(y_0 \in \partial X\)
	and consider the map \(I_{y_0}\) defined in
	\eqref{eqn:integrating-1-forms}.
	We know \(s\) and \(I_{y_0} \circ s^H \circ [\delta]\)
	differ by action of \(S^1\),
	but since both are equal to \(1\) at \(y_0\),
	thus \(s = I_{y_0} \circ s^H \circ [\delta]\).
	This implies that for any \(y \in \partial X\)
	and any embedded curve \(\gamma : [0,1] \to X\)
	with \(\gamma(0) = y_0\) and \(\gamma(1) = y\)
	we have that \(\exp\left(\int_{y_0}^{y} s^H([\delta(u)])\right) = 1\)
	and thus \(\int_{y_0}^{y_1} s^H([\delta(u)]) \in 2 \pi i \mathbb{Z}\).
	This proves that \(s^H\) is integral.

	Similarly, if \(\sigma\) is integral,
	then \(s = I_{y_0} \circ \sigma \circ [\delta]\)
	satisfies that \(\sigma = s^H\)
	and \(s(y) = \exp\left(\int_{y_0}^{y} s^H([\delta(u)])\right) = 1\)
	for any \(y \in \partial X\).
\end{proof}

To find a~consistent way of \emph{twisting} the boundary values of \(1\)\hyp{}forms
we consider ways to ``undo'' the action of \(\GaugeSplit[s](X)\)
on the boundary ``in a~linear fashion''.
\begin{definition}
	[gauge twisting]
	\label{def:gauge-twisting}
	We call a~homomorphism 
	\begin{math}
		\tau:H^1(X; i \mathbb{R}) \to \GaugeHarm(\partial X) 
		\simeq (S^1)^{\pi_0(\partial X)}
	\end{math}
	a~\emph{gauge twisting} for \(s\) if the composition
	\begin{equation*}
		\pi_0\GaugeHarm(X)
		\simeq
		H^1(X; 2 \pi i \mathbb{Z})
		\hookrightarrow 
		H^1(X; i \mathbb{R})
		\xrightarrow{\tau}
		\GaugeHarm(\partial X)
	\end{equation*}
	agrees with the action of the split gauge group on the boundary,
	\( R \circ s: \pi_0\GaugeHarm(X) \to \GaugeHarm(\partial X)\).
\end{definition}
Continuous homomorphisms from a~vector space
to \(S^1\) correspond to linear functionals on the vector space.
Thus, every such twisting \(\tau\) comes from a~linear map
\(d\tau : H^1(X;i\mathbb{R}) \to H^0(\partial X; i \mathbb{R}) \)
and \(\tau = \exp \circ (d \tau)\).
We utilize it to prove the existence of gauge twistings for \(s\),
and one could use it to classify all possible gauge twistings
for \(s\).
Actually, every such homomorphism \(\tau\) is a~gauge twisting for some \(s\),
but we do not use this fact in this article.
\begin{lemma}
	[existence of gauge twistings]
	\label{lem:existence-of-gauge-twistings}
	For a~given gauge splitting  \(s\) there exists a~gauge twisting \(\tau\).
\end{lemma}
\begin{proof}
	Since \(\pi_0 \GaugeHarm(X)\) is free,
	we can lift the map
	\( R \circ s: H^1(X; 2 \pi i\mathbb{Z})
	\simeq \pi_0 \GaugeHarm(X) \to \GaugeHarm(\partial X) 
	\simeq (S^1)^{\pi_0(X)}\)
	to a~homomorphism
	\(\widetilde{\tau} : H^1(X; 2 \pi i \mathbb{Z}) 
	\to H^0(\partial X; i\mathbb{R})\):
	\begin{equation*}
		\begin{tikzcd}
			 & H^0(\partial X; i\mathbb{R}) \arrow{d}{\exp} \\
			\pi_0 \GaugeHarm(X) \arrow{r}{R \circ s} 
			\arrow[dashed]{ru}{\widetilde{\tau}} 
			 & (S^1)^{\pi_0(\partial X)},
		\end{tikzcd}
	\end{equation*}
	and this extends to a~map
	\(\widetilde{\tau} : H^1(X; i \mathbb{R}) \to H^0(\partial X; i\mathbb{R})\) 
	by linearity.
	Taking \(\tau = \exp \circ \widetilde{\tau}
	: H^1(X; i \mathbb{R}) \to (S^1)^{\pi_0(\partial X)}\)
	gives a~gauge twisting for \(s\).
\end{proof}

With \(\tau\) in hand, there is a~way of defining
a~twisting on the whole Coulomb slice,
enabling us to finally define the twisted restriction maps.
\begin{definition}
	[twisted restriction map]
	\label{def:twisted-restriction-map}
	We define the \emph{Coulomb slice twisting}
	\begin{math}
		\tau_{CC}: \SForms[CC]{1}{X} \to \GaugeHarm(\partial X)
	\end{math}
	associated to \(\tau\)
	to be the composition
	\begin{align*}
		\SForms[CC]{1}{X}
		&\xrightarrow{\proj_{L^2,\perp}}
		\mathcal{H}^1_D(X)
		\simeq H^1(X, \partial X; \mathbb{R})
		\\&\xrightarrow{\iota_X^\ast}
		H^1(X; \mathbb{R})
		\xrightarrow{\tau}
		(S^1)^{\pi_0(\partial X)}
		\simeq \GaugeHarm(\partial X).
	\end{align*}

	We define the \emph{twisted restriction map}
	\begin{equation*}
		R_\tau : \CoulFour{CC}{X} 
		\to \CoulThree{\partial X}
	\end{equation*}
	by the formula
	\begin{math}
		R_\tau (A,\Phi) = (R(A), \tau_{CC}(A-A_0) R(\Phi)).
	\end{math}
\end{definition}

\begin{remark}
	What is of importance for defining the twisted restriction maps 
	is the map
	\(\tau_{CC} : i \Omega^1_{s}(X) \to (S^1)^{\pi_0(\partial X)}\).
	The extension of \(\tau_{CC}\) to the whole of
	\(i \Omega^1_{CC}(X)\) is artificial:
	it does not undo the action of
	\(\GaugeHarmId(X)\) on the boundary
	as one might expect.

	With more work, 
	including a~choice of a~based gauge group
	\( \GaugeHarm_o(X) \subset \GaugeHarm(X) \)
	(such that \( \GaugeHarm(X) / \GaugeHarm_o(X) \simeq S^1\))
	and a~more general twisting,
	one could work with the full \(i \Omega^1_{CC}(X)\) 
	and then quotient by the action of \(\GaugeHarm_o(X)\).
	However, this would introduce unnecessary complications.
\end{remark}

These twisted restriction maps
are indeed invariant under \(\GaugeSplit[s](X)\).
\begin{lemma}
	[twisted restriction map is invariant under split gauge group]
	\label{lem:twisted-restriction-invariance}
	Let \(\tau\) be a~gauge twisting for \(s\).
	For any \( (A,\Phi) \in \CoulFour{CC}{X} \)
	(resp. \( (a,\phi) \in \TangCoulFour{X} \))
	and \(u \in \GaugeSplit[s](X)\)
	we have
	\begin{equation*}
		R_\tau(u(A,\Phi)) = R_\tau(A,\Phi)
	\end{equation*}
	(resp. \(R_\tau(u(a,\phi)) = R_\tau(a,\phi)\)).
\end{lemma}

\begin{proof}
	Since \(u \in \GaugeHarm(X)\), 
	we have \(\iota_{\partial X}^\ast(u^{-1}du) = 0\),
	so \(R(A - u^{-1} du) = R(A)\).
	
	It remains to prove
	\begin{equation*}
		\tau_{CC}(A - A_0 - u^{-1}du) R(u \Phi) = \tau_{CC}(A-A_0) R(\Phi)
	\end{equation*}
	but that is equivalent to
	\begin{equation*}
		(\tau_{CC}(-u^{-1}du) R(u)) 
		\tau_{CC}(A - A_0) R(\Phi) = \tau_{CC}(A-A_0) R(\Phi)
	\end{equation*}
	so it suffices to prove
	\(\tau_{CC}(u^{-1}du) = R(u)\).
	Since \(s\) splits 
	\eqref{eqn:harmonic-gauge-sequence},
	we have
	\(u = s([u])\), where \( [u] \in \pi_0\GaugeHarm(X)\)
	is the homotopy class of \(u\).
	So we have to prove
	\(\tau_{CC}(u^{-1}du) = R\circ s([u])\).
	This follows directly from \autoref{def:gauge-twisting}
	of the twistings
	and \autoref{def:twisted-restriction-map}
	of the twisted restriction map,
	since \(u^{-1}du \in \mathcal{H}^1_D(X)\)
	and the isomorphism 
	\(\pi_0\GaugeHarm(X) \simeq H^1(X; 2 \pi i \mathbb{Z})\)
	is given by
	\( [u] \mapsto [u^{-1}du]\).
\end{proof}

We conclude these sections by showing that
choosing \(\tau\) is essentially equivalent to
choosing an~integral splitting.
In general, one can restrict themselves to considering
integral splittings without any twisting at all.

\begin{proposition}
	[twistings are integral splittings]
	\label{prop:twistings-are-integral-splittings}
	Let \(\tau\) be a~twisting for \(s\).
	Then there is an~integral splitting \(s_{\mathbb{Z}}\)
	and an~equivariant diffeomorphism
	\begin{equation*}
		F_{s,\tau} : \CoulFour{s}{X} \to \CoulFour{s_{\mathbb{Z}}}{X}
	\end{equation*}
	such that
	\(R \circ F_{s,\tau} = R_\tau\).
\end{proposition}
\begin{proof}
	Every function \(f \in \mathcal{H}(X)\) is determined by its 
	restriction to the boundary
	\(f|_{\partial X}\),
	which is locally constant.
	We thus have the exact sequence
	\begin{equation*}
		0 \to H^0(\partial X; 2 \pi i \mathbb{Z})
		\to H^0(\partial X; i \mathbb{R}) \simeq i \mathcal{H}(X)
		\xrightarrow{f \mapsto \exp(f|_{\partial X})} \GaugeHarm(\partial X)
		\to 0
	\end{equation*}
	as well as
	\begin{equation*}
		0 \to H^0(X; 2 \pi i \mathbb{Z})
		\to i \mathcal{H}(X)
		\xrightarrow{\exp} \GaugeHarmId(X)
		\to 0
	\end{equation*}
	and from these two it follows that
	\begin{equation}
		\label{eqn:gauge-restriction-ses}
		0 \to H^0(\partial X; 2 \pi i \mathbb{Z}) / H^0(X; 2 \pi i \mathbb{Z})
		\to \GaugeHarmId(X)
		\xrightarrow{\cdot |_{\partial X}} \GaugeHarm(\partial X)
		\to 0
	\end{equation}
	is exact.
	Since the group to the left is discrete
	it follows that there exists a~unique lift \(\tilde{\tau}\)
	of \(\tau\) to \(\GaugeHarmId(X)\):
	\begin{equation*}
		\begin{tikzcd}
			& \GaugeHarmId(X) \arrow{d}{\cdot|_{\partial X}} \\
			H^1(X; i \mathbb{R}) \arrow[dashed]{ru}{\tilde{\tau}}
			\arrow{r}{\tau}
			& \GaugeHarm(\partial X)
		\end{tikzcd}
	\end{equation*}
	We define 
	\begin{equation*}
		s_{\mathbb{Z}}([u]) = (\tilde{\tau}([u]))^{-1} \cdot s([u])
	\end{equation*}
	for any \( [u] \in \pi_0 \Gauge(X) \simeq H^1(X; 2 \pi i \mathbb{Z})\),
	and
	\begin{equation*}
		F_{s,\tau} = F_{(\tilde{\tau})^{-1}}
	\end{equation*}
	using the construction of \(F_\nu\) 
	of \autoref{prop:equivalence-of-slices}.
	This gives an~equivariant diffeomorphism
	from \(\CoulFour{s}{X}\) to \(\CoulFour{s_{\mathbb{Z}}}{X}\).

	The equality \(R \circ F_{s,\tau} = R_\tau\) 
	follows from the construction.
\end{proof}

Even though the spaces \(\CoulFour{s_{\mathbb{Z}}}{X}\)
and \(\CoulFour{s_{\mathbb{Z}}'}{X}\) are equivariantly diffeomorphic
by \autoref{prop:equivalence-of-slices},
the corresponding restriction maps differ by a~twist.
Thus, \textit{a priori} we cannot get rid of the choice of a~splitting.
However, this is not relevant to most 
of the applications because for connected boundary
there is no choice to make.
\begin{lemma}
	[uniqueness of integral splittings]
	\label{lem:uniqueness-of-integral-splittings}
	If \(\partial X \neq \varnothing\) is connected, 
	there exists exactly one integral splitting \(s_\mathbb{Z}\).
\end{lemma}
\begin{proof}
	In this case, the restriction map
	\(\GaugeHarmId(X) \to \GaugeHarm(\partial X)\)
	is an~isomorphism
	(cf. \eqref{eqn:gauge-restriction-ses}).
	Therefore for each element \(\pi_0 \Gauge(X)\)
	there exists exactly one representative \(u \in \GaugeHarm(X)\)
	such that \(u|_{\partial X} = 1\).
\end{proof}

\subsection{Seiberg\hyp{}Witten moduli in split Coulomb slice}
\label{sec:Seiberg-Witten-moduli-in-split-gauge}

We conclude this section by defining the \emph{Seiberg\hyp{}Witten moduli spaces},
the main object of study of this article.
We also prove they only depend on the choice of \(s_{\mathbb{Z}}\)
associated to \(s\) and \(\tau\).

Thanks to \autoref{lem:split-gauge-compatible},
we can define the following.
\begin{definition}
	[moduli spaces on \(4\)\hyp{}manifolds with boundary]
	\label{def:moduli-spaces-on-4-manifolds-with-boundary}
	We define the moduli spaces in split slice:
	\begin{equation*}
		\SWModuliFree{s}{X}
		=  \left\{ (A,\Phi) \in \CoulFour{s}{X}
			\middle| \SW(A,\Phi) = 0 \right\},
	\end{equation*}
	\begin{equation*}
		\SWModuli{s}{X}
		= \quotient{\SWModuliFree{s}{X}}{\GaugeSplit[s](X)}.
	\end{equation*}
\end{definition}
We also define a~version of the moduli space using the full double Coulomb slice,
\begin{equation*}
	\SWModuliFree{CC}{X}
	=  \left\{ (A,\Phi) \in \CoulFour{CC}{X}
	\middle| \SW(A,\Phi) = 0 \right\},
\end{equation*}
which will be utilized in some of the proofs.

From \autoref{lem:twisted-restriction-invariance} it follows that
\begin{corollary}
	There is a~well\hyp{}defined restriction map
	\begin{align*}
		R_\tau: \SWModuli{s}{X} & \to \CoulThree{\partial X}.
	\end{align*}
\end{corollary}

A~direct consequence of \autoref{prop:twistings-are-integral-splittings}
is
\begin{corollary}
	[dependence on twistings]
	\label{cor:dependence-on-twistings}
	Given \(s\) and \(\tau\), there is an~integral splitting \(s_{\mathbb{Z}}\)
	and a~diffeomorphism
	\begin{equation*}
		F_{s,\tau} : \SWModuli{s}{X} \to \SWModuli{s_{\mathbb{Z}}}{X}
	\end{equation*}
	such that
	\(R \circ F_{s,\tau} = R_\tau\).
\end{corollary}

\section{Properties of moduli spaces}
\label{sec:properties-of-moduli-spaces}

In this section we prove that
(\namedref{thm:semi-infinite-dimensionality-of-moduli-spaces}):
\begin{itemize}
	\item the moduli spaces
		of solutions to the Seiberg\hyp{}Witten equations on \(X\) are Hilbert manifolds,
	\item the restriction map to the boundary is ``semi\hyp{}infinite'',
		i.e., Fredholm in the negative direction and compact in the positive
		direction,
	\item if \(\partial X\) is disconnected, 
		restriction to a~single boundary component
		has dense differential.
\end{itemize}
This is done by analyzing the properties of the linearized Seiberg\hyp{}Witten
operator \(D\SW\).
We start by investigating an~extended version of this
operator, \(\widetilde{D\SW}\). 
The reason is that the standard
Atiyah\hyp{}Patodi\hyp{}Singer boundary value problem
as well as the elliptic theory developed in
\autoref{sec:analytical-preparation} can be directly applied to the study
of \(\widetilde{D\SW}\).
Our understanding of the gauge action  
(\autoref{sec:split-Coulomb-slice-on-4-manifolds})
will allow us to transfer these properties to \(D\SW\).

\subsection{Extended linearized SW operator}
\label{sec:properties-of-the-extended-linearized-SW-operator}
Here we apply the Atiyah\hyp{}Patodi\hyp{}Singer boundary value problem
to an~extended version of the linearized Seiberg\hyp{}Witten operator,
\(\widetilde{D\SW}\).
The properties we prove are the direct analogues of the properties
of \(D\SW\) which are proved in the next section.

\begin{definition}
	[extended linearized SW operator]
	\label{def:extended-DSW}
	We define the extended linearized Seiberg-Witten operator
	\begin{align*}
		\widetilde{D\SW}_{(A,\Phi)}
		:& \TangFour{X}\\
		\to & 
		L^2(X; i \mathbb{R}) \oplus \SWDomain{X} 
	\end{align*}
	by adding a~component related to the linearization of gauge action:
	\begin{equation*}
		\widetilde{D\SW}_{(A,\Phi)}(a,\phi)
		= (d^\ast a, D\SW_{(A,\Phi)}(a,\phi)).
	\end{equation*}
\end{definition}

In order to study the Atiyah\hyp{}Patodi\hyp{}Singer boundary value problem
we need to introduce the appropriate operator on the boundary
and consider its Calder\'on projector.
Denote \(Y = \partial X\)
and define 
\begin{align*}
	\tilde L : i \Omega^1(Y) \oplus \Gamma(S_Y) \oplus i \Omega^0(Y)
	\to& i \Omega^1(Y) \oplus \Gamma(S_Y) \oplus i \Omega^0(Y),
	\\
	\tilde L(b,\psi,c) =& (\star d b - d c, \Dirac_{B_0} \psi, - d^\ast b).
\end{align*}
This is a~first-order self-adjoint elliptic operator.
Denote by \(\widetilde{H}^+(Y,\mathfrak{s})\)
(resp. \(\widetilde{H}^-(Y,\mathfrak{s})\)
the closure of the span of positive (resp. nonpositive)
eigenspaces of \(\widetilde{L}\)
in \( L^2_{1/2}(i \Omega^1(Y) \oplus \Gamma(S_Y) \oplus i \Omega^0(Y))\),
and by \(\widetilde{\proj}^\pm\) the projection onto
\(\widetilde{H}^\pm(Y,\mathfrak{s})\) along \(\widetilde{H}^\mp(Y,\mathfrak{s})\).
The proof of the following proposition follows a~standard argument;
we briefly recall it to set up the stage for the proofs in the rest of this
section.
\begin{proposition}
	[semi-infinite-dimensionality of \(\widetilde{D\SW}\)]
	\label{prop:extended-DSW-semi-inf-dim}
	\hfill\\The operator
	\begin{align}
		\label{eqn:extended-DSW-APS}
		\widetilde{D\SW}_{(A,\Phi)} \oplus \widetilde{\proj}^- R
		:&\, \TangFour{X}  \\
		\nonumber
		\longrightarrow &\, L^2(X;i\mathbb{R}) \oplus \SWDomain{X} 
		\oplus \widetilde{H}^-(Y,\mathfrak{s})
	\end{align}
	is Fredholm of index
	\begin{equation}
		\label{eqn:extended-DSW-APS-index}
		2 \ind_{\mathbb{C}} \Dirac_{A_0} +\, b_1(X) - b^+(X) - b_1(Y) - 1.
	\end{equation}
	Moreover, the positive part of the restriction map
	from the kernel of \(\widetilde{D\SW}\),
	\( \widetilde{\proj}^+ R : \ker\mleft(\widetilde{D\SW}_{(A,\Phi)}\mright)
	\to \widetilde{H}^+(Y,\mathfrak{s}) \),
	is compact.
\end{proposition}

\begin{proof}
	We can write
	\begin{math}
		\widetilde{D\SW}_{(A,\Phi)} = \widetilde{D} + \widetilde{K}
	\end{math}
	where
	\begin{equation*}
		\widetilde{D}(a,\phi) = (d^+ a, \Dirac_{A_0} \phi, d^\ast a)
	\end{equation*}
	and
	\begin{equation*}
		\widetilde{K}(a,\phi) = (0, -\rho^{-1}(\phi \Phi^\ast + \Phi \phi^\ast)_0,
		\rho(a) \Phi + \rho(\diff{A}) \phi).
	\end{equation*}
	As explained in \cite{Kha2015}*{Proposition 3.1},
	applying the Atiyah-Patodi-Singer boundary value problem
	\cite{KMbook}*{Theorem 17.1.3}
	to \(\widetilde{D}\) proves that
	\(\widetilde{D}\oplus \widetilde{\proj}R\)
	is Fredholm with index equal to \eqref{eqn:extended-DSW-APS-index}.
	Furthermore, 
	\cite{KMbook}*{Theorem 17.1.3} implies that
	for any bounded sequence \((u_i) \subset \TangFour{X}\)
	such that \( (\widetilde{D}(u_i)) \) is Cauchy,
	the sequence \( (\widetilde{\proj}^+R u_i) \) is precompact.

	The operator \(\widetilde{K}\) is compact by \autoref{thm:multiplication}.
	Since \(\widetilde{D\SW}_{(A,\Phi)} = \widetilde{D} + \widetilde{K}\),
	thus \(\widetilde{D\SW}_{(A,\Phi)}\) is Fredholm with the same index 
	as \(\widetilde{D}\).
	Moreover, since \(\widetilde{K}\) is compact,
	for any sequence \( (u_i) \subset \ker
	\mleft(\widetilde{D\SW}_{(A,\Phi)}\mright)\)
	we can choose a~subsequence such that
	the sequence of \( \widetilde{D}(u_i) = - \widetilde{K}(u_i)\)
	is convergent, thus Cauchy.
	By what we proved in the previous paragraph,
	the sequence \( (\widetilde{\proj}^+R u_i) \) is precompact.
	This shows that
	\( \widetilde{\proj}^+ R : \ker\mleft(\widetilde{D\SW}_{(A,\Phi)}\mright)
	\to \widetilde{H}^+(Y,\mathfrak{s}) \)
	is compact.
\end{proof}

The proof of surjectivity 
utilizes both of the analytical results of
\autoref{sec:analytical-preparation} 
(cf. \cite{Lip2008}*{Theorem \(2\)}).
\begin{proposition}
	[surjectivity of \(\widetilde{D\SW}\)]
	\label{prop:extended-DSW-surjective}
	The operator 
	\begin{math} \widetilde{D\SW}_{(A,\Phi)} \end{math}
	is surjective.
\end{proposition}
\vphantom{.}\\
\vspace{-2em}
\begin{proof}
	Assume
	\begin{math} \widetilde{D\SW}_{(A,\Phi)} \end{math}
	is not surjective.
	\autoref{prop:extended-DSW-semi-inf-dim} implies
	its image is closed,
	so there is \(0 \neq \tilde v \in \mathcal{V}(X,\hat{\mathfrak{s}})
	\oplus L^2(i \Omega^0(X))\)
	orthogonal to \(\im \widetilde{D \SW}\).
	Recall that \(\tilde K\) is a~certain multiplication
	by \( p = (A - A_0, \Phi) \in L_1^2( i \Omega^1(X) \oplus \Gamma(S^+_X)) \).
	Let \(X^\ast = X \cup ( [0,\infty) \times Y )\)
	with cylindrical metric on the end,
	and extend the spinor bundle \(S_X\)
	to \(S_{X^\ast}\) which is cylindrical on ends.
	Extend \(p\) to 
	\(p^\ast \in L^2_1( i \Omega^1(X^\ast) \oplus \Gamma(S^+_{X^\ast}))\) 
	in an~arbitrary way
	and \(\tilde v\) to 
	\(\tilde v^\ast \in \mathcal{V}(X^\ast,\hat{\mathfrak{s}})
	\oplus L^2(i \Omega^0(X^\ast))\)
	by zero on \( [0,\infty) \times Y\).
	We have
	\begin{equation*}
		\langle \tilde v^\ast, (\tilde D + \tilde K) (w)
		\rangle_{L^2(X^\ast)}
		= 
		\langle \tilde v, \widetilde{D\SW}_{(A,\Phi)}(w|_X)
		\rangle_{L^2(X)}
		= 0
	\end{equation*}
	for any \(w \in \mathcal{TC}(X^\ast, \hat{\mathfrak{s}})\).
	Therefore \(\tilde v^\ast\) is a~weak solution to
	\( \tilde D^\ast \tilde v^\ast + \tilde K^\ast \tilde v^\ast = 0\)
	where \(\tilde D^\ast, \tilde K^\ast\) are formal adjoints
	of \(\tilde D, \tilde K\), respectively.
	The map \(\tilde K^\ast:L^2_1\to L^2\) is compact by 
	\autoref{thm:multiplication}.
	Thus from the \namedref{thm:low-regularity} it follows that
	\(\tilde v^\ast \in L^2_1(i \Omega^+(X^\ast) \oplus \Gamma(S^-_{X^\ast}) 
	\oplus i \Omega^0(X^\ast)) \)
	and it is a~solution to 
	\( (\tilde D^\ast + \tilde K^\ast) \tilde v^\ast = 0 \).
	Furthermore, \autoref{cor:unique-cont-dirac-4d}
	implies that 
	\(\tilde v^\ast = 0\) and therefore \(v = 0\).
	Thus, by contradiction, we have proved that \(D \SW\)
	is surjective.
\end{proof}

Finally, we focus on the density of the restriction map
from the kernel of \(\widetilde{D\SW}\) to one boundary component.
The proof of this proposition utilizes some of the ideas 
we have just seen (cf. \cite{Lip2008}*{Lemma 5}).
\begin{proposition}
	[density of moduli on one boundary component]
	\label{prop:extended-DSW-moduli-boundary-dense}
	Assume the boundary \(Y\)
	has at least two connected components
	and let \(Y_0 \subset Y\) be any one of these components.
	Then the restriction
	\begin{equation*}
		R : \ker\widetilde{D\SW}_{(A,\Phi)} \to \TangThree{Y_0}
	\end{equation*}
	is dense.
\end{proposition}

\begin{proof}
	Assume, by contradiction,
	that it is not dense and choose a~nonzero element
	\(v \in \TangThree{Y_0}\)
	which is \(L^2_{1/2}\)\hyp{}perpendicular to its image.
	Since \(\widetilde{D\SW}_{(A,\Phi)}\)
	is surjective by
	\autoref{prop:extended-DSW-surjective},
	the map
	\begin{align*}
		\widetilde{D\SW}_{(A,\Phi)} \oplus \Pi_v R :& 
		\TangFour{X} \\
		\to& L^2(X;i\mathbb{R}) \oplus \SWDomain{X}
		\oplus \mathbb{C} v
	\end{align*}
	has finite\hyp{}dimensional cokernel,
	where \(\Pi_v\) is the \(L^2_{1/2}\)\hyp{}projection
	onto \(v \in \TangThree{Y_0}\).

	From the definition of \(v\) it follows that
	\( \widetilde{D\SW}_{(A,\Phi)} \oplus \Pi_v R \)
	is not surjective, since otherwise there would be
	an~element \(w \in \TangFour{X}\)
	which solves \(\widetilde{D\SW}_{(A,\Phi)}(w) = 0\)
	such that \(\Pi_v(R(w)) \neq 0\).
	Therefore, we can pick an~element \( (a,v) \) 
	which is orthogonal to the image of 
	\( \widetilde{D\SW}_{(A,\Phi)} \oplus \Pi_v R \),
	where \(a \in L^2(X;i\mathbb{R}) \oplus \SWDomain{X}\).

	As in the proof of \autoref{prop:extended-DSW-surjective}, 
	attach a~cylindrical end along \(Y\) 
	to get \(X^\ast = X \cup [0,\infty) \times Y \).
	Extend \(a\) to \(a^\ast\) by \(0\) on \( [0,\infty) \times Y \).
	Since \(\tilde K\) is a~certain multiplication by
	\( p = (A - A_0,\Phi) \in L^2_1( i \Omega^1(X) \oplus \Gamma(S^+_X) )\),
	extend \(p\) to 
	\(p^\ast \in L^2_1(i \Omega^1(X^\ast) \oplus
	\Gamma(S^+_{X^\ast}))\) in an~arbitrary way to get 
	\(\tilde K\) defined on \(X^\ast\).
	By \autoref{thm:multiplication},
	\(\widetilde{K}\) is a~compact operator \(L^2_1 \to L^2\).

	Since \((a,v) \perp \im \mleft(\widetilde{D\SW}_{(A,\Phi)}
	\oplus \proj_v R \mright)\),
	we get that \(a^\ast\) is a~weak solution to
	\((\tilde D^\ast + \tilde K^\ast) a^\ast = 0\)
	on the interior of \(X^\ast_1 = X^\ast \setminus ([0,\infty) \times Y_0)\).
	Take any compact set \(C \subset X^\ast_1\)
	and a~smooth bump function \(\eta : X^\ast \to [0,1]\)
	such that \(\eta|_C = 1\)
	and \(\supp \eta \subset X^\ast_1\).
	It follows that \( (D^\ast+K^\ast)(\eta a) = \rho^\ast(d \eta) a \in
	L^2(X^\ast) \).
	Therefore by the \namedref{thm:low-regularity}
	we get that \(\eta a \in L^2_1(X^\ast)\),
	and in particular \(a \in L^2_1(\mathring{C})\).
	Varying \(C\) we obtain
	\(a^\ast \in L^2_{1,loc}(X^\ast_1)\).
	Since \(a^\ast \equiv 0\)
	on \([0,\infty) \times (Y \setminus Y_0) \subset X_1^\ast\),
	\autoref{cor:unique-cont-dirac-4d}
	implies that \(a^\ast = 0\) on \(X^\ast_1\),
	so \(a=0\) on \(X\).

	This is a~contradiction since we can extend
	\(v \in \TangThree{Y_0}\)
	to \(\tilde v \in \TangFour{X}\) 
	such that \(R(\tilde v) = v\).
	Then
	\begin{align*}
		0 =& 
		\langle (\widetilde{D\SW}_{(A,\Phi)}(\tilde v), \Pi_v R \tilde v),
		(a, v) \rangle_{L^2_{1/2}}
		\\ =& 
		\langle (\widetilde{D\SW}_{(A,\Phi)}(\tilde v), v),
		(0, v) \rangle_{L^2_{1/2}}
		\\ =&
		0 + \langle v,v\rangle_{L^2_{1/2}}
	\end{align*}
	implying \( \lVert v\rVert_{L^2_{1/2}} = 0\), 
	which contradicts
	the assumption that \(v \neq 0\).
\end{proof}

\subsection{Regularity and semi-infinite-dimensionality of moduli spaces}
\label{sec:regularity-and-semi-infinite-dimensionality-of-moduli-spaces}
We turn our focus to the operator \(D\SW\).
The main difficulty in transferring the results of 
\autoref{sec:properties-of-the-extended-linearized-SW-operator}
from \(\widetilde{D\SW}\) to \(D\SW\)
is the presence of the Coulomb condition
on both \(X\) and \(\partial X\).
The split gauge condition introduces an~additional twist to the story.

A~key fact is that the differential of the gauge group action at \(e\)
(cf. \autoref{lem:gauge-action-4d})
preserves the kernel of \(D\SW\)
at a~solution \( (A,\Phi) \).
\begin{lemma}
	\label{lem:gauge-algebra-action}
	Assume \(\Dirac_A \Phi = 0\).
	Then for any \(f \in L^2_2(i \Omega^0(X))\) we have
	\( D\SW_{(A,\Phi)}(df,-f \Phi) = 0\).
\end{lemma}
\begin{proof}
	We compute:
	\begin{align*}
		(\hat D + \hat K)&(df,-f \Phi)
		\\
		=& ( \rho^{-1}( f \Phi \Phi^\ast + \Phi (f \Phi)^\ast)_0,
		\rho(df) \Phi - \Dirac_{A_0} (f \Phi) - \rho(\diff{A}) f \Phi)
		\\ =&
		( \rho^{-1}( f \Phi \Phi^\ast - \Phi f \Phi^\ast)_0,
		\rho(df) \Phi - \rho(df) \Phi - f(\Dirac_{A_0}\Phi + \rho(\diff{A})\Phi) )
		\\ =&
		(0, - f \Dirac_A \Phi)
		\\ =& (0,0),
	\end{align*}
	which finishes the proof.
\end{proof}

Following the idea of \cite{Kha2015}*{Proposition 3.1},
we deduce the semi\hyp{}infinite\hyp{}dimensionality
and compute the index of \(D\SW\) from
\autoref{prop:extended-DSW-semi-inf-dim}.
These methods will be utilized to prove further results in this section,
too.

\begin{proposition}
	[semi-infinite-dimensionality of \(D\SW\)]
	\label{prop:DSW-semi-inf-dim}
	The operator
	\begin{align}
		\label{eqn:DSW-APS}
		D_{(A,\Phi)}\SW \oplus \proj^- R
		:& \TangCoulFour[s]{X} \\
		\nonumber
		& \to \SWDomain{X} \oplus H^-(Y,\mathfrak{s})
	\end{align}
	is Fredholm of index
	\begin{equation}
		\label{eqn:DSW-APS-index}
		2 \ind_{\mathbb{C}} + b_1(X) - b^+(X) - b_1(Y).
	\end{equation}
	Moreover, the restriction
	\( \proj^+ R : \ker\mleft(D_{(A,\Phi)}\SW\mright)
	\to H^+(Y,\mathfrak{s}) \)
	is compact, where 
	\( \ker\mleft(D_{(A,\Phi)}\SW\mright) \subset \TangCoulFour[s]{X} \).
\end{proposition}
\begin{proof}
	Firstly, we compare the respective polarizations.
	Recall that the decomposition of \(\TangThree{Y}\)
	into \(\widetilde{H}^+(Y,\mathfrak{s}) \oplus \widetilde{H}^-(Y,\mathfrak{s})\)
	is given by the eigenspaces of \(\widetilde{L}\).
	Decomposing 
	\begin{equation*}
		i \Omega^1(Y) \oplus \Gamma(S_Y) \oplus i \Omega^0(Y)
		= \mleft(i \Omega^1_{cC}(Y) \oplus \Gamma(S_Y)\mright)  
		\oplus \mleft(\Omega^1_C(Y) \oplus i \Omega^0(Y)\mright)
	\end{equation*}
	we see that \(\widetilde{L}\) decomposes as
	\( \begin{pmatrix} \star d & 0 \\ 0 & \Dirac_{B_0} \end{pmatrix} 
	\oplus \begin{pmatrix} 0 & -d \\ -d^\ast & 0 \end{pmatrix} \).
	Denote by \(\proj^\pm_1\) the spectral projections of
	\(\begin{pmatrix}
		0 & -d \\
		-d^\ast & 0
	\end{pmatrix}\)
	in \(L^2_{1/2}(i \Omega^1_C(Y) \oplus i \Omega^0(Y))\).
	It follows that \( \widetilde{\proj}^\pm = \proj^\pm \oplus \proj^\pm_1\).

	This is enough to prove the statement about compactness.
	Indeed, the map
	\(\proj^+R : \ker\mleft( D_{(A,\Phi)}\SW \mright)
	\to H^+(Y,\mathfrak{s})\)
	is just the composition
	\begin{equation*}
		\ker\mleft( D_{(A,\Phi)}\SW \mright)
		\hookrightarrow
		\ker\mleft( \widetilde{D\SW}_{(A,\Phi)} \mright) 
		\xrightarrow{\widetilde{\proj}^+R}
		\widetilde{H}^+(Y,\mathfrak{s})
		\xrightarrow{\proj^+} 
		H^+(Y,\mathfrak{s})
	\end{equation*}
	where the map in the middle is compact by
	\autoref{prop:extended-DSW-semi-inf-dim}.

	For Fredholmness, denote by 
	\begin{equation*}
		\Pi_C : L^2_{1/2}(\Omega^1_C(Y) \oplus \Omega^0(Y))
		\to L^2_{1/2}(\Omega^1_C(Y) \oplus \Omega^0(Y))
	\end{equation*}
	the projection onto \(L^2_{1/2}(\Omega^1_C(Y)) \oplus \mathbb{R}^{\pi_0(Y)}\)
	(where \(\mathbb{R}^{\pi_0(Y)} \subset \Omega^0(Y)\)
	is the space of locally constant functions)
	along \(\{0\} \oplus L^2_{1/2}(\Omega^0_0(Y))\)
	(where \( \Omega^0_0(Y) = \{ f \in \Omega^0(Y) | \forall_i \int_{Y_i} f = 0 \}\)).
	This projection can be used to define the split Coulomb slice
	for the orthogonal splitting \(s_\perp\) 
	(\autoref{def:orthogonal-splitting}).
	Precisely, we have
	\begin{equation*}
		\TangCoulFour[s_\perp]{X} = \ker\mleft(d^\ast, \proj_C\mright).
	\end{equation*}
	
	Khandhawit \cite{Kha2015}*{Proposition 3.1}
	proves that \(\im \Pi_1^-\) and \(\ker \Pi_C\) are complementary,
	and then \cite{KMbook}*{Proposition 17.2.6} implies that
	\begin{align}
		\label{eqn:extended-DSW-with-modified-boundary-condition}
		&\widetilde{D\SW}_{(A,\Phi)} \oplus (\Pi^- R) \oplus (\Pi_C R):
		\TangFour{X}
		\\&\nonumber
		\to \SWDomain{X}
		\oplus \SForms[C]{1/2}{Y} \oplus \mathbb{R}^{\pi_0(Y)}
		\oplus H^-(Y,\mathfrak{s})
	\end{align}
	is Fredholm since \eqref{eqn:extended-DSW-APS} is Fredholm.
	Since \(\Pi_C|_{\im \Pi_1^-}\) is an~isomorphism
	onto \(\im \Pi_C\), therefore the proof of \cite{KMbook}*{Proposition
	17.2.6} implies that the index of 
	\eqref{eqn:extended-DSW-with-modified-boundary-condition}
	is the same as the index of 
	\eqref{eqn:extended-DSW-APS}.

	The following lemma is a~simple exercise in Fredholm theory.
	\begin{lemma}
		\label{lem:Fredholm-pullback}
		Let \((F,G):H \to A \oplus B\) be Fredholm.
		Then \(\tilde F = F|_{\ker G}: \tilde H = \ker G \to A\) is Fredholm
		and has index equal to \(\ind( (F,G) ) + \dim \coker G \).
	\end{lemma} 
	It implies that the operator
	\( D_{(A,\Phi)}\SW \oplus (\Pi^- R) \)
	is Fredholm as a~map
	\(\ker( \star d \oplus (\Pi_C R)) = 
	\TangCoulFour[s_\perp]{X}
	\to 
	\SWDomain{X} \oplus H^-(Y,\mathfrak{s})\)
	and has index 
	\begin{align*}
		\ind&\mleft(\widetilde{D\SW}_{(A,\Phi)} \oplus (\Pi^- R) \oplus (\Pi_C R)\mright) 
		+ \dim \coker( \star d \oplus (\Pi_C R))
		\\=& [ 2 \ind_{\mathbb{C}} \Dirac_{A_0}^+
		- b_0(X) + b_1(X) - b^+(X) - b_1(Y) ] + b_0(X)
		\\=&  2 \ind_{\mathbb{C}} \Dirac_{A_0}^+
		+ b_1(X) - b^+(X) - b_1(Y).
	\end{align*}
	Thus, we have proven the Proposition for a~particular splitting,
	\(s=s_\perp\).

	For any splitting \(s\),
	the inclusion \(\Omega^1_{s}(X) \hookrightarrow \Omega^1_{CC}(X)\)
	is of codimension 
	\(\dim(H^1(X, Y; \mathbb{R})) 
	- \dim(H^1(X; \mathbb{R}) = b_0(Y) - 1\).
	Therefore, the split double Coulomb slice
	\(\TangCoulFour[s]{X}\)
	is a~finite\hyp{}dimensional subspace of 
	the ``full'' double Coulomb slice
	\(\TangFour{X}\).
	Therefore the Fredholmness of \eqref{eqn:DSW-APS}
	for \(s = s_\perp\)
	implies the Fredholmness of
	\begin{align}
		\label{eqn:DSW-APS-double}
		D_{(A,\Phi)}\SW \oplus \proj^- R
		:& \TangCoulFour{X} \\
		\nonumber
		& \to \SWDomain{X} \oplus H^-(Y,\mathfrak{s})
	\end{align}
	and this, in turn, implies the Fredholmness of
	\eqref{eqn:DSW-APS} for any splitting \(s\).
	Moreover, the index of \eqref{eqn:DSW-APS-double}
	is equal to \(b_0(Y) - 1\)
	plus \eqref{eqn:DSW-APS-index}
	for any splitting \(s\), finishing the proof.
\end{proof}

We turn to deducing the surjectivity of \(D\SW\)
from the surjectivity of \(\widetilde{D\SW}\).
\begin{proposition}
	[surjectivity of \(D\SW\)]
	\label{prop:DSW-surjective}
	For any gauge splitting \(s\), the differential
	\begin{equation*}
		D_{(A,\Phi)}\SW
		: \TangCoulFour[s]{X} \to \SWDomain{X}
	\end{equation*}
	is surjective.
\end{proposition}
\begin{proof}
	We will prove a~stronger statement,
	that this extended differential
	together with the exact part of the restriction to the boundary
	\begin{align}
		\label{eqn:DSW-and-boundary-exact-surjective}
		(\widetilde{D\SW}_{(A,\Phi)}, 
		\proj_d R, \proj_{V_s}) 
		&=
		(D\SW_{(A,\Phi)}, d^\ast, \proj_d R, \proj_{V_s}) : \\
		\nonumber
		\TangFour{X}&
		\to \\
		\nonumber
		\SWDomain{X}
		&\oplus L^2(i \Omega^0(X))
		\oplus \SForms[C]{1/2}{Y}
		\oplus V_s
	\end{align}
	is surjective,
	where \(\proj_d : \SForms{1/2}{Y} \to \SForms[C]{1/2}{Y}\)
	is the projection along \( \SForms[cC]{1/2}{Y} \)
	and \(\proj_{V_s} : \Omega^1(X) \to V_s = (\im s)^\perp\)
	is the orthogonal projection.
	The Proposition will follow since
	\(\TangCoulFour[s]{X} = \ker(d^\ast, \proj_d R, \proj_{V_s})\).

	\autoref{prop:extended-DSW-surjective} implies that
	\(\widetilde{D\SW}_{(A,\Phi)}
		: \TangFour{X}
		\to
		\SWDomain{X} \)
	is surjective.
	To prove surjectivity of \eqref{eqn:DSW-and-boundary-exact-surjective} 
	it thus remains to prove that
	\(\Pi_d R: \ker\mleft( \widetilde{D\SW}_{(A,\Phi)}\mright) \to
	\SForms[C]{1/2}{Y}\)
	is surjective and that
	\(\Pi_{V_s}: \ker \mleft(\widetilde{D\SW}_{(A,\Phi)}, \Pi_d R\mright) \to V_s\)
	is surjective.

	We prove that 
	\(\Pi_d R|_{\ker \widetilde{D\SW}_{(A,\Phi)}}\)
	is surjective.
	Take any \(g \in L^2_{3/2}(i \Omega^0(Y))\)
	representing a~given element
	\(dg \in L^2_{1/2}( i \Omega^1_C(Y))\).
	Take the unique \(f \in L^2_2(i \Omega^0(X))\)
	such that \(\Delta f = 0\) and \(f|_Y = g\).
	Then \(\Pi_d R(df, -f \Phi) = df|_Y = dg\).
	The required surjectivity follows since
	\( (df,-f \Phi) \in \ker \widetilde{D\SW}_{(A,\Phi)}\),
	which follows from \(d^\ast d f = \Delta f = 0\) and
	\autoref{lem:gauge-algebra-action}.

	Similarly we prove
	\(\Pi_{V_s}|_{\ker \mleft(\widetilde{D\SW}_{(A,\Phi)}, \Pi_d R\mright)}\)
	is surjective.
	By \eqref{eqn:decomposing-harmonics} and definition of \(V_s\),
	the orthogonal projection \(d(\mathcal{H}(X))) \to V_s\)
	is an~isomorphism and therefore for any
	\(v \in V_s\) there is
	\(f \in \mathcal{H}(X)\) such that \(\proj_{V_s}(df)=v_s\).
	Moreover
	\(\widetilde{D\SW}(df,-f\Phi) = 0\) and
	\(\proj_d R(df,-f\Phi) = 0\), as wished.
\end{proof}

Finally, we prove the density of the restriction map to a~connected component
of \(Y\).
\begin{proposition}
	[density of moduli on one boundary component]
	\label{prop:DSW-moduli-boundary-dense}
	The restriction
	\begin{equation*}
		R : \ker\mleft(D_{(A,\Phi)}\SW\mright) \to \TangCoulThree{Y_0}
	\end{equation*}
	to a~connected component \(Y_0 \subset Y\)
	is dense,
	where we consider \( \ker\mleft(D_{(A,\Phi)}\SW\mright) 
	\subseteq \TangCoulFour[s]{X} \).
\end{proposition}

\begin{proof}
	Take any
	\( (b,\psi) \in \TangCoulThree{Y_0} 
	\subset \TangThree{Y_0}\).
	It follows from \autoref{prop:extended-DSW-moduli-boundary-dense}
	that we can take a~sequence
	\( (\tilde a_k, \phi_k) \in \ker \mleft( \widetilde{D\SW}_{(A,\Phi)} \mright) \)
	such that \( R_{Y_0}(\tilde a_k, \phi_k) \to (b,\psi)\)
	in \(\TangCoulThree{Y_0}\).
	Using the decomposition \eqref{eqn:CC-exact-decomposition}
	we can write \(\tilde a_k = a_k + d f_k\)
	for \(f_k \in i \Omega^0(X)\) and \(a_k \in \Omega^1_{s}(X)\).
	Since \(Y_0\) is connected,
	we can change \(f_k\) by a~constant to obtain
	\(\int_{Y_0} f_k = 0\).
	Decomposing \(R_{Y_0}(\tilde a_k) = R_{Y_0}(a_k) + R_{Y_0}(d f_k)\)
	we get that \(R_{Y_0}(a_k) \to b\) and \(R_{Y_0}(d f_k) \to 0\).
	This together with \(\int_{Y_0} f_k = 0\)
	implies that \(f_k|_{Y_0} \to 0\)
	and therefore
	\begin{align*}
		(b,\psi)
		=&
		\lim_{k \to \infty} R_{Y_0}(\tilde a_k, \phi_k)
		=
		\lim_{k \to \infty} R_{Y_0}(a_k,\phi_k + i f_k \Phi)
	\end{align*}
	which finishes the proof because
	\(a_k \in \SForms[s]{1}{X}\)
	and
	\(D\SW(a_k,\phi_k + i f_k \Phi) = D\SW(a_k+i df_k, \phi_k) = 0\)
	by \autoref{lem:gauge-algebra-action}.
\end{proof}

The results of \autoref{prop:DSW-semi-inf-dim},
\autoref{prop:DSW-surjective} and 
\autoref{prop:DSW-moduli-boundary-dense}
can be summarized as follows.
\moduli

\section{Gluing along a~boundary component}
\label{sec:gluing-along-a-boundary-component}
This section is devoted to the proof of the main result of this article,
the \namedref{thm:composing-cobordisms}, which relates the moduli spaces
of solutions on \(X_1\), \(X_2\) and \(X = X_1 \cup_{Y} X_2\),
where \(Y\) is a~rational homology sphere,
oriented as a~component of the boundary of \(X_1\).
Under the identification
\(\CoulThree{Y} \simeq \Configuration[cC]{-Y}{\overline{\mathfrak{s}}}\)
there are \(S^1\)\hyp{}equivariant twisted restriction maps
\(R_{\tau_i,Y} : \SWModuli[i]{s_i}{X_i} \to \CoulThree{Y}\),
where \(\mathfrak{s} = \hat{\mathfrak{s}}|_Y\).
One can expect the fiber product
\begin{align*}
	\SWModuli[1]{s_1}{X_1} 
	&\times_{Y} \SWModuli[2]{s_2}{X_2} =
	\\
	= \{ 
	&(A_1, \Phi_1, A_2, \Phi_2) \in 
	\SWModuli[1]{s_1}{X_1} \times \SWModuli[2]{s_2}{X_2} 
	\\
	&| R_{\tau_1,Y}(A_1, \Phi_1) =
R_{\tau_2,Y}(A_2,\Phi_2)\}
\end{align*}
to be diffeomorphic to \(\SWModuli{s}{X}\),
and this turns out to be true.
One would also like to have this
map intertwine the twisted restriction maps
to \(\partial X\),
but this is a~bit too much:
the splittings and twistings 
\( (s,\tau)\), \((s_1,\tau_1)\), \((s_2,\tau_2)\),
need to enjoy certain compatibility, and even then
the restriction maps
may not match on the nose
but need to be homotoped to each other.
This reflects the fact that we did not quotient by the action of \(S^1\)
on the configuration spaces.

The proof utilizes the following fact which is of independent interest.
Let \(X' \subset X\) be a~submanifold,
the closure of which is contained in the interior of \(X\).
Then the restriction map from \(\SWModuliFree{CC}{X}\)
to \(L^2_k\)\hyp{}configurations on \(X'\) is well\hyp{}defined and \textit{smooth}
for any \(k \geq 0\).
Well\hyp{}definedness follows from a~standard argument,
but proving the smoothness of this map turns out to be
a~surprisingly delicate task which we tackle in 
\autoref{sec:smoothness-of-restrictions}.

The same strategy should work to prove smoothness of restriction
maps to interior submanifolds
for other types of moduli spaces appearing in gauge theory,
e.g., for the space of anti\hyp{}self\hyp{}dual connections on 
\(G \hookrightarrow P \to X\).
The key is the ellipticity of the equations
together with the gauge fixing.

\subsection{Smoothness of restrictions}
\label{sec:smoothness-of-restrictions}
Assume \(X' \subset \mathring{X}\) is a~submanifold
with closure contained in \(\mathring X\).
The goal is to show (cf. \autoref{thm:restriction-is-smooth-for-solutions})
that the restriction map
\(R: \SWModuliFree{CC}{X} \to
	\Configuration[k]{X'}{\hat{\mathfrak{s}}|_{X'}}\)
is smooth for any \(k\),
where
\(\Configuration[k]{X'}{\hat{\mathfrak{s}}|_{X'}}
= \ConfFourFull[k]{A_0|_{X'}}{X'}
\).
We restrict ourselves to the case \(k=2\), 
but the same strategy may be used iteratively,
bootstraping the result to any \(k\), if needed.

Due to \autoref{thm:trace} we may assume, without loss of generality,
that \(X'\) is of codimension \(0\).
Since we require the closure of \(X'\) to be contained in the interior
\(\mathring X\), we may as well assume that \(X'\) is a~closed submanifold.

The following fundamental fact shows that any element in the image
of the restriction map is itself a~smooth configuration.
\begin{lemma}
	[interior smoothness of solutions]
	\label{lem:interior-smoothness-of-solutions}
\cite{KMbook}*{Lemma 5.1.5}
	Every \(\gamma \in \SWModuliFree{CC}{X}\)
	is smooth on \(\mathring{X}\).
\end{lemma}

We are ready to prove the main theorem of this subsection.
Note that the surjectivity assumption is satisfied whenever
\(\partial X \neq \varnothing\)
due to \autoref{prop:DSW-surjective}.
\begin{theorem}
	[restriction is smooth on solution sets]
	\label{thm:restriction-is-smooth-for-solutions}
	Assume that for any \((A,\Phi) \in \SWModuliFree{CC}{X}\)
	the operator \(D\SW_{(A,\Phi)}\) is surjective.
	Then the restriction map
	\(R: \SWModuliFree{CC}{X} \to
	\Configuration[k]{X'}{\hat{\mathfrak{s}}|_{X'}}\)
	is smooth.
\end{theorem}
\begin{proof}
	\newcommand{\SPC}{L^2_{2,X''}\bigl(i \Omega^1_{CC}(X) \oplus%
	\Gamma\mleft(S_X\mright)\bigr)}
	\newcommand{\SPCnorm}[1]{\bigl\|#1\bigr\|_{2,X''}}
	\newcommand{\SPConf}{(A_0,0)+\SPC}
	\newcommand{\SPDom}{L^2_{1,X''}\bigl(i \Omega^+(X) \oplus%
	\Gamma\mleft(S_X^+\mright)\bigr)}
	Choose a~compact codimension\hyp{}\(0\) submanifold
	\(X'' \subset \mathring{X}\) such that \(X' \subset \mathring{X}''\).
	Let us introduce an intermediate space
	\(\SPC\)
	defined as the completion of
	\(i \Omega^1_{CC}(X) \oplus \Gamma(S_{X}^+)\)
	with respect to the norm
	\( \SPCnorm{v} = \sqrt{\lVert \hat D v\rVert^2_{L^2_1(X'')} +
	\lVert v\rVert^2_{L^2_1(X)}}\),
	so that \(\SPC\) is a~Hilbert space.
	Define 
	\begin{equation*}
		\SWModuliFree{CC,X''}{X}
		= \SWModuliFree{CC}{X} \cap \mleft(\SPConf\mright),
	\end{equation*}
	the set of solutions to the Seiberg\hyp{}Witten equations
	in the corresponding configuration space.
	Since \(\tilde D\) is elliptic and \( \tilde D v = (\hat D v, 0) \),
	by \autoref{thm:garding} the restriction map
	\[\SPC \to L^2_2(i \Omega^1(X') \oplus \Gamma(S_{X'}^+))\]
	is continuous linear (thus smooth).
	It thus suffices to prove that 
	the moduli space \(\SWModuliFree{CC,X''}{X}\)
	is a~smooth submanifold of the configuration space
	\(\SPConf\)
	and that the identity map
	\(\mathrm{Id}_{2}: \SWModuliFree{CC}{X} 
	\to \SWModuliFree{CC,X''}{X}\)
	is well\hyp{}defined, continuous and smooth.
	
	Firstly, it is well\hyp{}defined by 
	\autoref{lem:interior-smoothness-of-solutions}.

	Secondly, we prove that
	\(\SWModuliFree{CC,X''}{X}\)
	is a~smooth submanifold of
	\(\SPConf\).
	Define
	\(\SPDom\)
	to be the Hilbert space obtained as the completion of
	\(i \Omega^+(X) \oplus \Gamma(S_X^+)\)
	with respect to the norm
	\( \lVert v\rVert_{L^2(X) \cap L^2_1(X'')} = \sqrt{\lVert v\rVert^2_{L^2(X)}
	+ \lVert v\rVert^2_{L^2_1(X'')}}\).
	By the \namedref{thm:IFT}, it suffices to prove the following Lemma.
	\begin{lemma}
		\label{lem:DSW-surjective-higher-regularity}
		The differential
		\begin{equation*}
			D\SW_{(A,\Phi)} : \SPC
			\to \SPDom
		\end{equation*}
		is surjective at each \( (A,\Phi) \in 
		\SWModuliFree{CC,X''}{X}\).
	\end{lemma}
	\begin{subproof}
		We assumed that
		the operator
		\(
		D\SW_{(A,\Phi)} 
		\)
		is surjective.
		Choose any \(w \in \SPDom\)
		and pick a~solution \(v \in \TangCoulFour{X}\)
		to \(D\SW_{(A,\Phi)}(v) = w\).
		Then \(\hat D v|_{X''} = - \hat K v |_{X''} + w|_{X''}\)
		is in \(L^2_1(X'')\)
		since \(w|_{X''} \in L^2_1(X'')\), \(v|_{X''} \in L^2_1(X'')\)
		and \(\hat K|_{X''}\) is a~certain multiplication by
		\( (A - A_0,\Phi)|_{X''} \), which is smooth
		by \autoref{lem:interior-smoothness-of-solutions}.
	\end{subproof}

	Finally, it remains to prove that the identity map
	\[\SWModuliFree{CC}{X}
	\to \SWModuliFree{CC,X''}{X}\]
	is smooth.
	We start by identifying the tangent spaces 
	at \(p = (A,\Phi)\); we have
	\begin{align*}
		T_p \SWModuliFree{CC}{X} 
		&= \ker D\SW_{p}, \\
		\qquad T_p \SWModuliFree{CC,X''}{X}
		&= \mleft(\ker D\SW_p\mright) \cap {\SPC}.
	\end{align*}
	\begin{lemma}
		[isometry of the tangent spaces]
		\label{lem:tangent-spaces-isometry}
		The identity map
		\(T_p \SWModuliFree{CC}{X}
		\to T_p \SWModuliFree{CC,X''}{X}\)
		is well\hyp{}defined, and an~isometry.
	\end{lemma}
	\begin{subproof}
		Take any \(v \in
		T_p \SWModuliFree{CC}{X}\);
		in particular, \(D\SW_{(A,\Phi)}(v)=0\).
		Well\hyp{}definedness follows since
		\( \hat D v|_{X''} = - \hat K v|_{X''}\)
		and as before, \(\hat K|_{X''}\) is a~multiplication by 
		a~smooth configuration on \(X''\).
		This also implies
		\begin{align*}
			\lVert v\rVert^2_{L^2_1(X)}
			\leq \SPCnorm{v}^2
			=& \lVert v\rVert^2_{L^2_1(X)} + \lVert \hat D v\rVert^2_{L^2_1(X'')}
			\\=& \lVert v\rVert^2_{L^2_1(X)} + \lVert \hat K v\rVert^2_{L^2_1(X'')}
			\\\leq& \lVert v\rVert^2_{L^2_1(X)} + C_p \lVert v\rVert^2_{L^2_1(X'')}
			\\\leq& (1+C_p) \lVert v\rVert^2_{L^2_1(X)} 
		\end{align*}
		which proves this map is an~isometry.
	\end{subproof}

	While the \(L^2(X)\) norm is not complete on \(\TangCoulFour{X}\), 
	the \(L^2(X)\)\hyp{}orthogonal complement \(H\) to
	\(T_p \SWModuliFree{CC}{X}\)
	is a~closed subspace of \(\TangCoulFour{X}\)
	such that \(H + T_p \SWModuliFree{CC}{X} = \TangCoulFour{X}\)
	and \(H \cap T_p \SWModuliFree{CC}{X} = \{0\}\),
	thus by the open mapping theorem 
	\[\TangCoulFour{X} = T_p\SWModuliFree{CC}{X} \oplus H.\]
	By the \namedref{thm:IFT}
	there is a~neighborhood \(U\) of \( p \) such that the affine projection
	\begin{equation*}
		U \cap \SWModuliFree{CC}{X} \to 
		U \cap \mleft( (A_0,0) + T_p \SWModuliFree{CC}{X}\mright) 
	\end{equation*}
	along \(H\) is a~diffeomorphism.
	Similarly, the \namedref{thm:IFT}
	implies that there is a~neighborhood \(V\) of \(p\)
	such that the affine projection
	\begin{equation*}
		U \cap \SWModuliFree{CC,X''}{X}
		\to U \cap \mleft( (A_0,0) + T_p \SWModuliFree{CC,X''}{X}\mright)
	\end{equation*}
	along \(H' = H \cap \SPC\) is a~diffeomorphism
	since \(H'\) is the \(L^2(X)\)\hyp{}orthogonal complement
	to \(T_p \SWModuliFree{CC,X''}{X}\).
	Thus the identity map
	\(\SWModuliFree{CC}{X}
	\to \SWModuliFree{CC,X''}{X}\)
	near \(p\) factors as
	\begin{align*}
		U \cap \SWModuliFree{CC}{X}
		& \to 
		U \cap \mleft( (A_0,0) + T_p \SWModuliFree{CC}{X}\mright)
		\\ & \xrightarrow{\mathrm{id}}
		V \cap \mleft( (A_0,0) + T_p \SWModuliFree{CC,X''}{X}\mright)
		\\ & \to 
		V \cap \SWModuliFree{CC,X''}{X}
	\end{align*}
	where the middle identity map is smooth by
	\autoref{lem:tangent-spaces-isometry}
	and the two other maps are smooth since they are parametrizations
	coming from the \namedref{thm:IFT}, as we just showed.
\end{proof}

\subsection{Proof of the gluing theorem}
\label{sec:proof-of-the-gluing-theorem}

We are ready to prove the gluing theorem.

Let \(X = X_1 \cup_{Y} X_2\)
with \(Y\) connected and \(b_1(\partial X_i) = 0\).
Let \(\hat{\mathfrak{s}}\) be a~\spinc{} structure
and \(A_0\) be a~reference \spinc{} connection on \(X\).
Let restrictions of \(\hat{\mathfrak{s}}\) be the \spinc{} structures
used to define configuration spaces
and the restrictions of \(A_0\) to be the reference connections.
Denote \(Y_1 = \partial X_1 \setminus Y\),
\(Y_2 = \partial X_2 \setminus Y\),
\(\hat{\mathfrak{s}}_i = \hat{\mathfrak{s}}|_{X_i}\).
Fix gauge splittings
\(s,s_1,s_2\) and twistings \(\tau, \tau_1, \tau_2\) on \(X, X_1, X_2\).

Denote by \(s_{\mathbb{Z}}, s_{1,\mathbb{Z}}, s_{2,\mathbb{Z}}\)
the associated integral splittings
given by \autoref{prop:twistings-are-integral-splittings}.
We say they are \emph{compatible} 
if \(s_{\mathbb{Z}}\) corresponds to \((s_{1,\mathbb{Z}}, s_{2,\mathbb{Z}})\)
under the following identification.
\begin{proposition}
	[integral splittings on a~composite cobordism]
	\label{prop:integral-splittings-on-a-composite-cobordism}
	There is a~canonical identification between the set of integral
	splittings on \(X\) and the set of pairs of integral splittings
	on \(X_1\) and \(X_2\).
\end{proposition}
\begin{proof}
	Choose an~integral splitting \(s\).
	Take any \(a \in \mathcal{H}^1_D(X)\).
	Denote by \(\tilde{a}_i\) its restriction to \(X_i\).
	For each \(i\) there is a~unique 
	\(f_i \in L^2_2(\Omega^0(X_i))\) 
	such that \(\Delta f_i = 0\),
	\(f_i|_Y = G_d \iota_Y^\ast(a)\),
	and \(f_i|_{Y_i} = 0\).
	The resulting
	\(a_i = \tilde{a}_i - df_i\)
	is in \(\mathcal{H}^1_D(X_i)\),
	thus we get a~map
	\begin{equation*}
		R_{\mathcal{H}} : \mathcal{H}^1_D(X) \to 
			\mathcal{H}^1_D(X_1) \times \mathcal{H}^1_D(X_2)
	\end{equation*}
	sending \(a\) to \( (a_1,a_2)\).
	Note that \(\mathcal{H}^1_D(X_i) \subset \Omega^1_{CC}(X_i)\)
	and therefore \(R_{\mathcal{H}}\)
	coincides with doing the gauge fixing of
	\autoref{lem:split-gauge-fixing-in-4d} 
	on both components, i.e.,
	\( a \mapsto (\proj_{CC}(a|_{X_1}), \proj_{CC}(a|_{X_2})) \).
	The cohomology class \([a] \in H^1(X; \mathbb{R})\)
	restricts to \( ([\tilde{a}_1],[\tilde{a}_2]) = ( [a_1], [a_2]) 
	\in H^1(X_1;\mathbb{R}) \times H^1(X_2;\mathbb{R})\).
	It follows that the composition
	\begin{align*}
		H^1(X_1;\mathbb{R}) 
		&\times H^1(X_2;\mathbb{R})
		\xrightarrow{\iota_\ast} H^1(X; \mathbb{R})
		\xrightarrow{s^H} \mathcal{H}^1_D(X) \to
		\\
		& \xrightarrow{R_{\mathcal{H}}} \mathcal{H}^1_D(X_1) \times \mathcal{H}^1_D(X_2)
		\to
		H^1(X_1;\mathbb{R}) \times H^1(X_2;\mathbb{R})
	\end{align*}
	is the identity, where \(\iota_\ast\)
	is the inverse of the restriction map
	\(H^1(X; \mathbb{R})\to
	H^1(X_1;\mathbb{R}) \times
	H^1(X_2;\mathbb{R})\)
	(invertible by Mayer\hyp{}Vietoris).
	We can thus choose \( (s_1^H,s_2^H) \)
	to be the composition \(R_{\mathcal{H}} \circ s^H \circ \iota_\ast\).
	Since \(s^H\) was integral, thus \(s_i^H\) are integral
	and by \autoref{prop:homological-classification-of-integral-splittings} 
	there exist unique integral splittings
	\(s_1, s_2\) inducing \(s_1^H, s_2^H\).

	On the other hand, given integral splittings \(s_1, s_2\)
	on \(X_1, X_2\), we can choose \(s^H\) to be 
	\( (R_{\mathcal{H}})^{-1} \circ (s_1^H,s_2^H) \circ (\iota_\ast)^{-1}\)
	and by \autoref{prop:homological-classification-of-integral-splittings}
	there is a~unique integral splitting \(s\) inducing \(s^H\).
\end{proof}

\autoref{prop:DSW-surjective} guarantees that whenever \(\partial X \neq \varnothing\), 
the moduli \(\SWModuli{s}{X}\) is a~smooth Hilbert manifold,
and the same follows for \(\SWModuli{s_i}{X_i}\).
We want to include the case when \(X\) is a~closed manifold,
when it is well\hyp{}known
that to achieve surjectivity one, in general, needs to 
perturb the metric on \(X\)
or perturb the Seiberg\hyp{}Witten equations.
Therefore, we \emph{assume} that for any \((A,\Phi) \in \SW^{-1}(0)\) 
the operator \(D\SW_{(A,\Phi)}\) is onto.

\begin{remark}
	The careful reader may notice that we do not assume any transversality
	of the maps \(R_{\tau_i,Y}\) 
	and thus the fiber product may not \textit{a~priori} be a~manifold.
	That it is a~manifold follows from the proof of the theorem.
	What is not proven here, but may be useful in other contexts,
	is that the transversality
	is indeed equivalent to \(D\SW_{(A,\Phi)}\) being surjective for all
	\((A,\Phi) \in \SWModuli{s}{X}\).
\end{remark}

\gluing
\begin{proof}
	By \autoref{cor:dependence-on-twistings}
	we can assume, without loss of generality,
	that \(s = s_{\mathbb{Z}}\) and \(s_i = s_{i,\mathbb{Z}}\),
	as well as \(\tau \equiv 1\) and \(\tau_i \equiv 1\).
	We thus drop \(\tau\)'s from the notation entirely.

	The plan is as follows. We will construct a~map
	\begin{equation*}
		F : \SWModuliFree{s}{X}
		\to \SWModuliFree[1]{s_1}{X_1} \times \SWModuliFree[2]{s_2}{X_2},
	\end{equation*}
	and a~homotopy
	\begin{equation*}
		H: \SWModuliFree{s}{X}
		\to \mleft( \CoulThree{Y_1} \mright) \times 
		\mleft( \CoulThree{Y_2} \mright) 
	\end{equation*}
	between \(R\) and 
	\( \mleft(R_{Y_1} \times R_{Y}\mright) \circ F\).
	Then we will prove 
	that \(F\) intertwines the actions of
	\(\GaugeSplit[s](X)\)
	and \(\GaugeSplit[s_1](X_1) \times \GaugeSplit[s_2](X_2)\),
	and that \(H\) is \(\GaugeSplit[s](X)\)\hyp{}invariant;
	thus, both \(F\) and \(H\) descend to
	\(\SWModuli{s}{X}\).
	Furthermore, we will prove \(F\) and \(H\)
	are continuous and smooth,
	and that \(F\) has image in the fiber product.
	Finally, we will show that \(F\)
	is a~smooth immersion onto
	\begin{math}
		\SWModuliFree[1]{s_1}{X_1} \times_{Y} \SWModuliFree[2]{s_2}{X_2}.
	\end{math}
	The \(S^1\)\hyp{}equivariance of \(F\) and \(H\) will be apparent
	from the construction.

	{\bf Step 1.} 
	We start by constructing \(F\),
	proving its smoothness and that its image lies in the fiber product.
	Take \((A,\Phi) \in \SWModuli{s}{X}\).
	We would like to simply restrict it to the components \(X_i\)
	and then put into split Coulomb slice.
	By \autoref{lem:split-gauge-fixing-in-4d} modulo \(S^1\)
	there a~unique way of doing that using a~contractible gauge
	transformation. Here we make a~different choice
	than in \autoref{lem:split-gauge-fixing-in-4d},
	requiring \(\int_{Y} f = 0\) instead of \(\int_X f = 0\).
	Precisely, for \(a \in \SForms{1}{X_i}\)
	choose \(\tilde{u}^s_a = e^{f_a}\)
	such that \(f_a \in L^2_2(i \Omega^0(X_i))\), \(\int_{Y} f_a = 0\)
	and
	\begin{math}
		a 
		- \left(\tilde{u}^s_a\right)^{-1} d \tilde{u}^s_a 
		\in \SForms[s]{1}{X_i}.
	\end{math}
	Define
	\begin{equation*}
		F(A,\Phi) 
		= \mleft( \tilde{u}^s_{(A-A_0)|_{X_1}} (A,\Phi)|_{X_1},
		\tilde{u}^s_{(A-A_0)|_{X_2}} (A,\Phi)|_{X_2} \mright).
	\end{equation*}
	Since 
	\(A-A_0\) was in double Coulomb slice,
	thus \((A-A_0)|_{X_1}\) (resp. \((A-A_0)|_{X_2}\)) is already coclosed
	on \(X_1\) (resp. \(X_2\)) and on \(Y_1\) (resp. \(Y_2\)),
	moreover \((A-A_0)|_{X_1}\) and \((A-A_0)|_{X_2}\) agree on \(Y\).
	Denoting the restriction to \(Y\) by \(b_{(A,\Phi)} =
	\iota^\ast_{Y}(A-A_0)\),
	\autoref{thm:restriction-is-smooth-for-solutions}
	together with \autoref{thm:trace}
	imply that \(b_{(A,\Phi)}\) is smooth and depends smoothly on
	\((A,\Phi) \in \SWModuliFree{s}{X}\) as an~element of \(L^2_{3/2}\).
	Notice that \(f_i = f_{(A-A_0)|_{X_i}}\) 
	can be decomposed as \(f_i = f_i^{\partial} + f_i^s\),
	where \(f_i^\partial\)
	are the unique solutions to
	\begin{align}
		\label{eqn:fi}
		f_i^\partial|_{Y_i} = 0, \quad f_i|_{Y} = g, &\quad \Delta f_i = 0,
	\end{align}
	where \(g = G_d \Pi_d b_{(A,\Phi)}\)
	depends smoothly on \((A,\Phi)\) as an~element of \(L^2_{5/2}\).
	Moreover, the map \(g \mapsto f_i^\partial\) is linear and continuous
	as a~map \(L^2_{s+1/2} \to L^2_{s+1}\) for \(s \geq 0\),
	so \(f_i^\partial\) depend smoothly on \((A,\Phi)\) as~elements of \(L^2_3\).
	Furthermore, since 
	\(a|_{X_i} - df_i^\partial \in L^2_1(i\Omega^1_{CC}(X_i))\),
	\(f_i^s\) are the unique elements of \(\mathcal{H}^1_D(X_i)\)
	such that \(a|_{X_i} - df_i^\partial - d f_i^s \in L^2_1(i \Omega^1_s(X_i))\)
	and \(f_i|_Y = 0\).
	By \autoref{rmk:continuous-gauge-fixing-within-double-Coulomb-slice},
	\(f_i^s\) depend continuously on \(a|_{X_i} - d f_i^\partial \in L^2_1\).
	Which proves that \(f_i \in L^2_3\) depend continuously
	and linearly on \(g \in L^2_{5/2}\),
	thus depend smoothly on \(g\), so they depend smoothly on \( (A,\Phi)\).

	This establishes the well\hyp{}definedness and smoothness of \(F\).
	That its image lies in the fiber product follows directly from the construction.

	{\bf Step 2.} We proceed to constructing \(H\).
	By \eqref{eqn:fi}, the functions
	\(f_i\) are locally constant on \(Y_i\).
	We also have
	\begin{equation*}
		R_{Y_i}(F(A,\Phi)) = e^{f_i|_{Y_i}} \cdot
		R_{Y_i}(A,\Phi)
	\end{equation*}
	by the construction of \(F\). Thus we can define
	\begin{equation*}
		H(A,\Phi,t) = \mleft( e^{t f_1|_{Y_1}}, e^{t f_2|_{Y_2}} \mright) 
		R(A,\Phi)
	\end{equation*}
	which at \(t=0\) coincides with 
	\(R(A,\Phi)\)
	and at \(t=1\) coincides with 
	\( (R_{Y_1},R_{Y_2}) \circ F(A,\Phi)\).

	{\bf Step 3.} We now investigate the equivariance of \(F\)
	under the actions of gauge groups.
	Let \(u \in \GaugeSplit[s](X)\).
	From \autoref{lem:split-gauge-fixing-in-4d} if follows that
	there is exactly one contractible 
	\(\tilde{u}_i  = e^{\tilde{f}_i}\in \GaugeId(X_i)\)
	which puts \(-u^{-1}du\) into \(\SForms[s]{1}{X_i}\)
	with \(\int_{Y} \tilde{f}_i = 0\).
	Equivalently, this is the unique \(\tilde{u}_i = e^{\tilde{f}_i} \in \GaugeId(X_i)\)
	with \(\int_{Y} \tilde{f}_i = 0\)
	such that \(\tilde{u}_i u|_{X_i} \in \GaugeSplit[s_i](X_i)\).
	Define \(u_i = \tilde{u}_i u|_{X_i}\).
	Since \(\tilde{u}_i\) is contractible,
	thus \([\iota_{X_i}^\ast(u^{-1}du)] = [u_i^{-1}du_i]\)
	in \(H^1(X_i; 2 \pi i \mathbb{Z})\).
	Therefore the map
	\(u \mapsto (u_1,u_2)\)
	provides the~canonical isomorphism
	\begin{equation}
		\label{eqn:gauge-group-isomorphism}
		\GaugeSplit[s](X) \simeq \GaugeSplit[s_1](X_1) \times
		\GaugeSplit[s_2](X_2)
	\end{equation}
	which agrees with the isomorphism coming from the Mayer\hyp{}Vietoris
	sequence \(H^1(X; \mathbb{Z})
	\simeq H^1(X_1;\mathbb{Z}) \times H^1(X_2;\mathbb{Z})\).
	From the construction of \(F(A,\Phi)\)
	it follows that \((u|_{X_1}^{-1},u|_{X_2}^{-1}) F(u(A,\Phi))\)
	differs from \(F(A,\Phi)\)
	exactly by the factor of \( (e^{\tilde{f}_1},e^{\tilde{f}_2}) \).
	Thus \( (u_1^{-1},u_2^{-1}) F(u(A,\Phi)) = F(A,\Phi)\).
	This proves that \(F\) commutes with the gauge group action
	as identified in \eqref{eqn:gauge-group-isomorphism}.

	{\bf Step 4.} We prove the invariance of \(H\) under \(\GaugeSplit[s](X)\).
	By the construction of \(H\), and since \(u|_{\partial X} = 1\)
	for \(u \in \GaugeSplit[s](X)\),
	we have
	\(H( u(A,\Phi), t) = (e^{t\tilde{f}_1|_{Y_1}},e^{t\tilde{f}_2|_{Y_2}}) H(A,\Phi,t)\)
	where \(\tilde{f}_i\) are as in the previous paragraph.
	Since \(R_{\mathcal{H}}(\mathcal{H}^1_D(X)) \in (\im s_1^H) \times(\im
	s_2^H)\) we get that \(R_{\mathcal{H}}(u^{-1}du)\) 
	is already in the split Coulomb slice
	on \(X_1\) and \(X_2\).
	Moreover, \(a_i = \iota^\ast_{X_i}(u^{-1} du)\)
	is coclosed on \(X_i\) and \(Y_i\).
	Since there are functions \(\hat{f}_i\) satisfying
	\begin{align*}
		\Delta \hat{f}_1 = 0, \quad \hat{f}_1|_{Y_1} = 0, & \quad \hat{f}_1|_Y = G_d \iota^\ast_Y a, \\
		\Delta \hat{f}_2 = 0, \quad \hat{f}_2|_{Y_2} = 0, & \quad
		\hat{f}_2|_{Y} = G_d \iota^\ast_Y a,
	\end{align*}
	the uniqueness of \(\tilde{f}_i\) implies \(\tilde{f}_i = \hat{f}_i\)
	and therefore \(\tilde{f}_i|_{Y_i} = 0\).
	Thus \(H(u(A,\Phi),t) = H(A,\Phi,t)\), as wished.

	{\bf Step 5.} We show that \(F\) is bijective onto the fiber product, 
	following the argument in \cite{Lip2008}.
	Let \( (A_i, \Phi_i) \in \SWModuliFree{s_i}{X_i}\)
	such that \( R_{Y}(A_1,\Phi_1) = R_{Y}(A_2,\Phi_2)\).
	These would give a~configuration on \(X\) 
	if the normal components of connections \(A_1\) and \(A_2\) agreed on
	\(Y\).
	Let \(h_1\,dt\) and \(h_2\,dt\) be the \(dt\)\hyp{}components of
	\( (A_1-A_0)|_{Y}\) and \( (A_2-A_0)|_{Y}\).
	We want to find harmonic functions \(f_i \in L^2_2(X_i;i \mathbb{R})\)
	such that
	\[ f_1|_{Y} = f_2|_{Y}, \]
	\[ \partial_t f_1|_{Y} + h_1 = \partial_t f_2|_{Y} + h_2, \]
	\[ f_1|_{Y_1} = 0, f_2|_{Y_2} = 0.\]
	Take a~tubular neighborhood \( [-\varepsilon,\varepsilon] \times Y
	\subset X\) of \(Y\).
	Let \(\{\phi_\lambda\}_{\lambda}\) be an eigenbasis for \(\Delta_{Y}\)
	and write \(h_2-h_1 = \sum_\lambda c_\lambda \phi_\lambda\).
	Since \(h_2 - h_1 \in L^2_{1/2}\), therefore
	\( \sum_\lambda \lambda^{1/2} |c_\lambda|^2 < \infty\)
	and thus the following are well\hyp{}defined as elements of
	\( L^2_2([-\varepsilon,\varepsilon] \times Y; i \mathbb{R})\):
	\begin{align*}
		g_1 &= \frac 1 2 
		\sum c_\lambda \lambda^{-1/2} e^{\lambda^{1/2} t} \varphi_\lambda,
		\\
		g_2 &= \frac 1 2
		\sum c_\lambda \lambda^{-1/2} e^{-\lambda^{1/2} t} \varphi_\lambda,
	\end{align*}
	which satisfies
	\(\partial_t g_1|_{Y} - \partial_t g_2|_{Y} = h_1-h_0\).
	Let \(\rho \in C^\infty(X;\mathbb{R})\) be a~bump function supported in
	\([-\varepsilon,\varepsilon] \times Y\) which is identically \(1\)
	in a~neighborhood of \(Y\).
	The configurations
	\(e^{\rho g_1}(A_1,\Phi_1)\) and \(e^{\rho g_2} (A_2,\Phi_2)\)
	patch to give a~\(L^2_1\) configuration \( (A',\Phi')\), but this is not necessarily
	in the Coulomb slice because \(\rho f_i'\) are not necessarily harmonic.
	Take \(f \in L^2_2(X;i \mathbb{R})\) such that
	\[f|_{Y_1} = 0, \quad f|_{Y_2} = 0,  \quad
	\Delta f = - \Delta(\rho g_1) - \Delta(\rho g_2).\]
	Denote
	\( (A'',\Phi'') = e^{f} (A',\Phi') \in \SWModuliFree{CC}{X} \).
	Finally, by
	\autoref{rmk:continuous-gauge-fixing-within-double-Coulomb-slice}
	we can continuously deform \((A'',\Phi'')\) to a~configuration \((A,\Phi)\)
	such that \(F(A,\Phi) = ( (A_1,\Phi_1), (A_2,\Phi_2) )\).

	{\bf Step 6.} 
	We need to prove that \(F^{-1}\) constructed in the previous step is continuous.
	Notice \( (g_1,g_2) \) as elements of \(L^2_2\)
	depend continuously on \(h_1-h_0 \in L^2_{1/2}\)
	which in turn depends continuously on \(A_1-A_0\) and \(A_2-A_0 \in L^2_{1/2}\).
	Moreover, \(f \in L^2_3\) depends continuously on
	\( (g_1,g_2) \in L^2_2 \).
	If the multiplication \(L^2_2 \times L^2_1 \to L^2_1\) was continuous
	on \(4\)\hyp{}manifolds then we would have shown that the map
	\( ((A_1,\Phi_1),(A_2,\Phi_2)) \to (A,\Phi)\) which we constructed 
	is continuous.
	Since it is not true, we need to show that \( e^{g_i} \Phi_i \in L^2_1\) 
	depends continuously on the initial configurations.
	
	We will prove it depends continuously on 
	\( h_2-h_1 \in L^2_{1/2}\) and \(\Phi_i \in L^2_1\).
	Let \( (A_1',\Phi_1')\) and \( (A_2',\Phi_2') \)
	be another choice of configurations,
	and denote the corresponding harmonic functions on
	\([-\varepsilon,0] \times Y\) and \( [0,\varepsilon] \times Y\)
	by \(g'_1\), \(g'_2\).
	Then
	\begin{align*}
		\| e^{g_1} \Phi_1 - e^{g'_1} \Phi_1' \|_{L^2_1}
		& \leq \| e^{g'_1} (\Phi_1 - \Phi_1') \|_{L^2_1}
		+ \| (e^{g_1} - e^{g'_1}) \Phi_1 \|_{L^2_1}
		\\ & 
		\leq C(\| e^{g'_1} \|_{L^\infty} + \| e^{g'_1} \|_{L^2_2})
		\| \Phi_1 - \Phi_1'\|_{L^2_1}
		\\ &
		+ C  \|e^{g_1}-e^{g'_1}\|_{L^2_2} 
		\| \Phi_1\|_{L^2_1} 
		\\ &
		+ C \|e^{g_1}-e^{g'_1}\|_{L^\infty([-\varepsilon,-\delta] \times Y)}
		\| \Phi_1\|_{L^2_1([-\varepsilon,-\delta] \times Y)}
		\\ & 
		+ C \|e^{g_1}-e^{g'_1}\|_{L^\infty([-\delta,0] \times Y)}
		\| \Phi_1\|_{L^2_1([-\delta,0] \times Y)}
	\end{align*}
	for any \(\delta\).
	One can choose \(\delta\) to have
	\( \|\Phi_1\|_{L^2_1([-\delta,0] \times Y}\) as small as one wants
	while \( \|e^{g_1}-e^{g'_1}\|_{L^\infty([-\delta,0] \times Y)} \leq 2\).
	Moreover, we have
	\begin{equation*}
		\|e^{g_1}-e^{g'_1}\|_{L^\infty([-\varepsilon,-\delta] \times Y)}
		\leq \|g_1 - g'_1\|_{L^\infty([-\varepsilon,-\delta] \times Y)}
		\leq C \| (h_1 - h_2) - (h_1' - h_2')\|_{L^2_{1/2}}
	\end{equation*}
	via a~direct computation (or by interior regularity estimates
	following from \autoref{thm:garding}).
	This finishes the proof that the inverse map is continuous.

	{\bf Step 7.} 
	Finally, we prove that the differential of \(F\) is invertible.
	Assume this is not the case,
	so that there exists \( (A,\Phi) \in \SWModuliFree{s}{X}\)
	such that for any \(\varepsilon>0\) there is
	\( (A',\Phi') \in \SWModuliFree{s}{X}\)
	such that \(0 < D = \| F (A,\Phi) 
	- F (A',\Phi') \|_{L^2_1} < 1\)
	and
	\( \| (A-A',\Phi-\Phi') \|_{L^2_1} \leq D \varepsilon\).
	We get that \( \|g-g'\|_{L^2_{5/2}} \leq C_{5/2} D \varepsilon\)
	and thus
	\begin{align*}
		\| f_i - f_i'\|_{L^2_3}
		&\leq C \| g - g'\|_{L^2_{5/2}}
		\\& \leq C C_{5/2} D \varepsilon.
	\end{align*}
	From this and \autoref{thm:multiplication} it follows that
	\( D = \| F(A,\Phi) - F(A',\Phi') \|_{L^2_1}
	\leq C'' \| (A-A',\Phi-\Phi') \|_{L^2_1} \leq C''D \varepsilon\)
	for some \(C''\) depending on \(\|\Phi\|_{L^2_1}\).
	Choosing \(\varepsilon = \frac{1}{1+C''}\)
	gives the desired contradiction.
\end{proof}

\appendix
\section{Appendix}

This section presents some standard analytical results
which are used repeatedly in the article.
We recall the G{\aa}rding inequality,
Sobolev multiplication and trace theorems
and the Implicit Function Theorem.


\begin{theorem}
	[Sobolev multiplication theorem]
	\label{thm:multiplication}
	Let \(M\) be a~manifold with compact boundary
	and cylindrical ends.

	Assume \(k, l \geq m\) and \( 1/p + 1/q \geq 1/r\)
	for \(p,q,r \in (1,\infty) \).
	Then the multiplication
	\[
		L_k^p(M) \times L_l^q(M) \to L_m^r(M)
	\]
	is continuous if any of these hold:
	\begin{itemize}
		\item \( (k-n/p) + (l-n/q) \geq m - n/r \)
			and both \( k - n/p < 0, l - n/q < 0\),
		\item \( \min(k-n/p, l-n/q) \geq m - n/r \)
			and either \( k - n/p > 0 \) or \( l - n/q > 0\),
		\item \( \min(k-n/p, l-n/q) > m - n/r\)
			and either \( k - n/p = 0 \) or \( l - n/q = 0\).
	\end{itemize}
	What is more, whenever it is continuous, it restricts to a~compact map
	on \( \{f\} \times L_l^q(M) \to L_m^r(M) \)
	provided that \( l > m \) and \( l - n/q > m - n/r \).
\end{theorem}
\begin{proof}
	We can reduce to the case of a~manifold without boundary
	by considering a~double of \(M\), 
	or by attaching cylindrical ends along \(\partial M\).

	To obtain continuity,
	combine \cite{Pal1968}*{Theorem 9.6}
	for compact manifolds together with
	\cite{KMbook}*{Theorem 13.2.2} for an~infinite cylinder.

	For the compactness statement on a~compact manifold,
	for each sequence \(g_i \xrightarrow{L^q_l(M)} g\)
	take any sequence \(f_j \in C^\infty(M)\)
	such that \(\|f - f_j\|_{L^p_k(M)} \leq \frac{1}{2^j}\).
	By continuity of multiplication
	there is a~constant \(C\) such that
	\begin{equation*}
		\| (f-f_j) g_i\|_{L^r_m(M)}
		\leq C \|f-f_j\|_{L^p_k(M)} \leq \frac{C}{2^j}.
	\end{equation*}
	Set \(a_{0,i} = i\) and define inductively \(a_{j,i}\)
	in the following manner.
	Take \(a_{j,i} = a_{j-1,i}\) whenever \(i < j\).
	Since \(f_j \in C^\infty(M)\), the sequence
	\(f_j g_i\) is bounded in \(L^q_l(M)\)
	and by Sobolev embedding is precompact in \(L^r_m(M)\).
	Thus, we can choose \( (a_{j,i})_i \)
	to be a~subsequence of \( (a_{j-1,i})_i \)
	satisfying
	\begin{equation*}
		\|f_j g_{a_{j,j}} - f_j g_{a_{j,i}}\|_{L^r_m(M)}
		\leq \frac{1}{2^j}
	\end{equation*}
	for any \(i \geq j\).
	We finally get, for \(i \geq 0\),
	\begin{align*}
		\| f g_{a_{j,j} - f g_{a_{j+i,j+i}}} \|_{L^r_m(M)}
		& \leq
		\| (f-f_{j}) (g_{a_{j,j} - g_{a_{j+i,j+i}}}) \|_{L^r_m(M)}
	  \\
	  & 
		+ \| f_{j} (g_{a_{j,j} - g_{a_{j+i,j+i}}}) \|_{L^r_m(M)}
		\\
		& \leq
		\frac{2 C}{2^j}
		+ \frac{1}{2^j}
	\end{align*}
	making \(f g_{a_{j,j}}\) a~Cauchy sequence in \(L^r_m(M)\), 
	as wished.
	
	Compactness for a~cylinder 
	is proved in \cite{KMbook}*{Theorem 13.2.2}.
\end{proof}
%

\begin{theorem}
	[Sobolev trace theorem]
	\label{thm:trace}
	Suppose \( N \subset M\) is a~closed compact 
	\(n-j\)\hyp{}dimensional smooth submanifold,
	\( 1 \leq p < \infty \), and \( k, l \geq 0\)
	with \( k - j/p \geq l > 0 \).
	Then the restriction map extends to continuous
	\[L^p_k(M) \to L^p_l(N). \]
\end{theorem}
\begin{proof}
	See \cite{Pal1968}*{Theorem 9.3}. 
\end{proof}

We will make use of the Implicit Function Theorem
in the following form (cf. \cite{Lan1993}).
\begin{iftheorem}
	[Implicit Function Theorem]
	\label{thm:IFT}
	Suppose \(A, B\) are Hilbert spaces
	and \(\mathcal{F} : A \to B\) is a~smooth map
	such that the derivative \(D_p \mathcal{F}\) at \(p\) 
	is surjective
	and that its kernel splits with \(C\) as a~complementary subspace,
	\(A = \ker\mleft( D_p \mathcal{F} \mright) \oplus C\).
	Let \(\proj : A \to \ker\mleft( D_p \mathcal{F} \mright) \) 
	denote the projection onto the kernel along \(C\).
	Then there are open neighborhoods 
	\(U \subset A\) of \(p\),
	\(V \subset B\) of \(\mathcal{F}(p)\) 
	and \(W \subset \ker\mleft( D_p \mathcal{F} \mright) \) of \(0\) 
	and a~smooth diffeomorphism
	\(\mathcal{G} : V \times W \to U\)
	such that
	\[V \times W \xrightarrow{\mathcal{G}} U \xrightarrow{(\mathcal{F},\proj)} 
	B \oplus \ker \mleft( D_p \mathcal{F} \mright) \]
	is the identity map
	(i.e., \(\mathcal{G} = \mathcal{F}^{-1}\)).

	In particular, the projection
	\begin{equation*}
		\proj : \mathcal{F}^{-1}(0) \cap U \to
		\ker\mleft( D_p \mathcal{F} \mright) 
	\end{equation*}
	along \(C\) is a~local diffeomorphism.
\end{iftheorem}

We will also utilize the G{\aa}rding inequality.
\begin{theorem}
	[G{\aa}rding inequality]
	\label{thm:garding}
	Let \(D\) be a~first\hyp{}order elliptic operator with smooth coefficients on
	a~compact manifold \(M\) (possibly with boundary)
	and \(M' \subset M\) be open with compact closure.
	Then there is a~constant \(C\) such that for any \(\gamma \in
	L^p_{k+1}(M)\) we have
	\[ \lVert \gamma\rVert_{L^p_{k+1}(M')} \leq C ( \lVert  D \gamma\rVert_{L^p_k(M)} +
	\lVert \gamma\rVert_{L^p_k(M)}).\]
\end{theorem}
\begin{proof}
	This follows from \cite{Shu1992}*{Appendix 1, Lemma 1.4}
	by extending \(D\) to the cylindrical\hyp{}end manifold 
	\(M^\ast = M \cup (\partial M) \times [0,\infty)\).
	Indeed, taking a~smooth bump function \(\rho\)
	such that \(\rho|_{M'} = 1\) and \(\rho|_{M^\ast \setminus M} = 0\)
	we get
	\begin{align*}
		\|\gamma \|_{L^p_{k+1}(M')} 
		& \leq
		\| \rho \gamma \|_{L^p_{k+1}(M^\ast)}
	 \\ &\leq
		C ( \| D(\rho \gamma) \|_{L^p_k(M^\ast)}
		+ \| \rho \gamma\|_{L^p_k(M^\ast)})
	 \\ &\leq
		C C'_\gamma ( \| D(\gamma) \|_{L^p_k(M)}
		+ \| \gamma\|_{L^p_k(M)})
	\end{align*}
	where the middle inequality follows from
	\cites{Shu1992}.
\end{proof}
\begin{remark}
	The author's understanding is that without the use of twistings
	there is (in general) no choice of
	\(s, s_1, s_2\) making \(F\) commute with the restriction
	maps on the nose.
	This problem does not show up in the construction
	of monopole Floer homology since there one quotients
	the moduli and configuration spaces by \(S^1\) after blowing up.
\end{remark}




\end{document}